\documentclass{amsart}
\usepackage[utf8]{inputenc}
\usepackage{setspace}
\usepackage[margin=1.25in]{geometry}
\usepackage{graphicx}
\graphicspath{ {./figures/} }
\usepackage{subcaption}
\usepackage{amsmath}
\usepackage{amssymb}
\usepackage{mathtools}
\usepackage{lineno}
\usepackage{amsfonts}
\usepackage{xfrac} 
\usepackage{graphicx} 
\usepackage{amsthm}
\usepackage[all]{xy}

\usepackage[pagebackref=true]{hyperref}
\usepackage{float}
\usepackage{braket}

\usepackage{multirow}
\usepackage{diagbox}
\usepackage{slashbox}

\usepackage{tikz-cd}
\usetikzlibrary{quotes}
\usetikzlibrary{cd}

\newtheorem{theorem}{Theorem}
\newtheorem{example}[theorem]{Example}
\newtheorem{definition}[theorem]{Definition}
\newtheorem{lemma}[theorem]{Lemma}

\newtheorem{remark}[theorem]{Remark}
\newtheorem{proposition}[theorem]{Proposition}

\newtheorem*{proposition*}{Proposition}
\newtheorem*{theorem*}{Theorem}
\newtheorem*{definition*}{Definition}

\begin{document}

\title{A Decomposition Theorem for Topological Semi-small Maps}

\author{Shahryar Ghaed Sharaf}

\date{\today}


        \begin{abstract}
         A continuous surjection $f: X \rightarrow Y$ is an almost covering map if there exists a nowhere dense closed set $R \subset Y$ with a finite decomposition $R = \bigsqcup_{\beta} T_{\beta}$ such that the restriction $f\vert_{f^{-1}(Y \setminus R)}$ is a covering map and, for each $\beta$, the restriction $f\vert_{f^{-1}(T_{\beta})}$ is a fiber bundle. We refer to an almost covering map as a topological semi-small map if, for each $\beta$, the fiber of $f\vert_{f^{-1}(T_\beta)}$ and the space $T_\beta$ satisfy the same dimension condition as for algebraic semi-small maps. We first study the topological properties of almost coverings and develop the necessary tools. In the next step, we restrict our attention to topological semi-small maps 
         and prove that, under the assumption that for each $\beta$ and each 
         trivializing neighborhood $U \subset T_{\beta}$ the connecting homomorphisms 
         in the Gysin sequence of the normal bundle of $f^{-1}(U) \subset X$ vanish 
         in suitable degrees, the derived direct image of the constant sheaf on the 
         domain admits a decomposition.
         \end{abstract}

\maketitle

\tableofcontents
	
     \section{Introduction}

     Let $f:X \longrightarrow Y$ be a proper surjective algebraic morphism with $X$ smooth and pure-dimensional. It is well known that there exists a Whitney stratification $Y = \bigsqcup_{\alpha} S_{\alpha}$ with respect to which $f$ is a stratified morphism; in particular, for each $\alpha$, the restriction $f^{-1}(S_{\alpha}) \longrightarrow S_{\alpha}$ is a locally trivial fibration (see e.g.~\cite{Verdier1976}). The algebraic morphism $f:X \longrightarrow Y$ is called semi-small if 
     \begin{align}\label{DimCond}
     2 \dim_{\mathbb{C}} (f^{-1}(s_{\alpha}))+ \dim_{\mathbb{C}}(S_{\alpha}) \leq \dim_{\mathbb{C}}(X)
     \end{align}
   for all $\alpha$ and $s_{\alpha} \in S_{\alpha}$. The semi-small morphisms have been studied by Borho and MacPherson in \cite{BorhoMacPherson1981} and \cite{BorhoMacPherson1983}. Let $S_{\alpha_{\operatorname{top}}}$ be the top stratum of $Y$, and let $s \in S_{\alpha_{\operatorname{top}}}$. From the semi-small condition, it follows that $\dim_{\mathbb{C}}(f^{-1}(s))=0$. Hence, the semi-small morphism $f$ restricted to the top stratum is topologically a covering map. A stratum $S$ is called relevant if the equality  $2 \dim_{\mathbb{C}} (f^{-1}(s))+ \dim_{\mathbb{C}}(S) = \dim_{\mathbb{C}}(X)$ holds. Furthermore, since $f$ is a proper map, the complex $\mathcal{L}_{S}:= (R^{n-\dim_{\mathbb{C}(S) }}f_{\ast} \underline{\mathbb{Q}}_{X})\vert_{S}$ is a local system. The algebraic semi-small morphisms are subject to the following decomposition theorem (see e.g. \cite{deCataldoMigliorini2002})
   \begin{theorem*}
   	Let $f : X \to Y$ be a proper surjective semi-small map .
   	Let $Y= \bigsqcup_{\alpha}S_{\alpha}$ be a Whitney stratification of $Y$ with respect to which $f$ is stratified, and denote by $Y_{\mathrm{rel}} \subset Y$ the set of relevant strata of $f$.
   	For each $S \in Y_{\mathrm{rel}}$, let $\mathcal{L}_S$ be the corresponding local system defined above.
   	There is a canonical quasi-isomorphism
   	\begin{align*}
   	Rf_* \mathbb{Q}_X[\dim_{\mathbb{C}}(X)] \simeq \bigoplus_{S \in Y_{\mathrm{rel}}} IC_{\overline{S}}(\mathcal{L}_S).
   	\end{align*}
   \end{theorem*}

  The existence of polarizable Hodge modules in the algebraic setting, as defined by Saito in \cite{Saito1985,Saito1988}, enables a decomposition theorem for a proper map $f:X \longrightarrow Y$ of complex algebraic varieties, where $X$ is pure-dimensional. More precisely, the BBDG decomposition theorem \cite{BBD} yields a non-canonical decomposition of $Rf_{\ast}IC_{X}$, the derived direct image of Deligne's sheaf complex.

   BBDG decomposition theorem is a fundamental structural result. It gives the geometric proof of the Kazhdan–Lusztig conjectures \cite{Soergel1990}. In the Langlands programme, Ngô in \cite{Ngo2010} exploited the existence of such decompositions to prove a support theorem for the Hitchin fibration, a key ingredient in the Fundamental Lemma. The theorem further underpins the Springer correspondence \cite{BorhoMacPherson1983}, the construction of character sheaves \cite{lusztig1985I}, and the geometric Satake equivalence \cite{MirkovicVilonen}. For further applications, the reader may consult \cite{deCataldo2017}.

   Generalizations of algebraic decomposition theorems to a broader topological setting, even for smooth topological manifolds fail, in general. For a smooth proper submersion $f:M \longrightarrow N$ of smooth manifolds, Ehresmann’s theorem, introduced in \cite{Ehresmann1951}, fully describes the local behavior of $f$. To be more precise, the map $f$ is a locally trivial smooth fiber bundle and all fibers
   are diffeomorphic. Hence, the sheaf $R^{k}f_{\ast}\underline{\mathbb{Q}}_{M}$ is a local system for each $k$. But the complex $Rf_{\ast} \underline{\mathbb{Q}}_{M}$ does not admit a decomposition in terms of these local systems in the bounded derived category $D^{b}_{c}(N)$, in general. A classical example is provided by the Hopf fibration $f : S^{3} \longrightarrow S^{2}$, whose local behavior is fully understood, yet the complex $Rf_{\ast}\underline{\mathbb{Q}}_{S^{3}}$ does not admit a decomposition in the above sense. Hence, in a more general topological setting, one needs more than just the local behavior of a map in order to have any hope of obtaining a decomposition theorem for the map.

    This work addresses the following question: Under which conditions, for a given map $f: X \longrightarrow Y$ which is a covering over a dense open subset, does the right derived direct image complex $Rf_{\ast}\underline{\mathbb{Q}}_{X}$ decompose in the derived category $D^{b}_{c}(Y)$, at least when $X$ is a topological manifold? Our goal is to answer this question in a general topological setting, without assuming any smooth or algebraic structure. In a recent work \cite{sharaf2025decomposition}, we developed a decomposition theorem for topological branched coverings. The principal objects of investigation in this work are the maps defined below, which constitute a generalization of branched coverings obtained by relaxing the finiteness condition over the branch locus.

        \begin{definition*}[Almost Covering Maps]
     	A continuous proper surjection $ f: X \longrightarrow Y$ between path-connected, locally compact, Hausdorff spaces $X$ and $Y$ is called an \textbf{almost covering map}
    	if 
    	\begin{enumerate}
    	\item there exists a nowhere dense closed set $R \subset Y$ such that the restriction $f \vert_{f^{-1}(Y \setminus R)} : f^{-1}(Y \setminus R) \longrightarrow
    	Y \setminus R$ is a finite covering map, and
    	\item there is a finite decomposition $R=\bigsqcup_{\beta}T_{\beta}$ such that the restriction $f \vert_{f^{-1}(T_{\beta})}:f^{-1}(T_{\beta}) \longrightarrow T_{\beta}$ is a fiber bundle for each $\beta$.
        \end{enumerate}
        The set $Y \setminus R$ is called the regular set of the almost covering map $f$, and the set $R$ is referred to as a fibered set of the map $f$.
        \end{definition*}
        In this work, we restrict our attention to the cases where $X$ is a topological manifold and $Y$ is a topological pseudomanifold. Furthermore, we assume that the inclusions $R \hookrightarrow Y$ and $f^{-1}(R) \hookrightarrow X$ are locally flat. In Section~\ref{TopAlCov}, we study the topology of almost covering maps; here, without any analytic or algebraic structure, many standard topological tools---for instance, the existence of stratifications with respect to which $f$ is stratified---are unavailable. Note that not all almost covering maps can be stratified. In particular, the decomposition of the stratified set $R:= \bigsqcup_{\beta}T_{\beta}$, given in the above definition, cannot be arbitrary. Proposition \ref{RefinedStra} delivers the stratification with respect to which we construct Deligne's sheaf complex, defined in \ref{DeSheaf}, in the decomposition theorem that we develop, assuming that the decomposition of the fibered set satisfies the conditions in Definition \ref{SutStrat}, which is of a technical nature.

        In Section \ref{DecomSec}, we first develop a decomposition theorem for a simple class of almost covering maps. To be more precise, assume that $f^{-1}(R)$ is a submanifold of $X$. As shown in \cite{Rourke1967} by Rourke and Sanderson, not every submanifold which is embedded piecewise linearly possesses a topological normal microbundle in the ambient space. As it turns out, the normal bundle of $f^{-1}(R)$ in $X$, if it exists, can be used to fully answer the question of whether the complex $Rf_{\ast}\underline{\mathbb{Q}}_{X}$ admits a decomposition in the derived category. Hence, we assume its existence. Note that in a finer topological setting, for example, when working with a Whitney stratification or a Thom--Mather stratification on $Y$, the induced stratification on $X$ is also Whitney or Thom--Mather. In particular, a stratum of a Whitney or Thom-Mather stratified space has a tubular neighborhood in the ambient space, and hence the previous assumption is not necessarily (see e.g. \cite{Mather2012} and \cite{mather1973stratifications}). The following proposition is the first main result of Section \ref{DecomSec}.

    \begin{proposition*}[Proposition \ref{DecoTheo1}]
    	Let $f : X \longrightarrow Y$ be an almost covering such that $X$ and the fibered set $R$ are closed topological manifolds, and $Y$ is a closed topological $n$-pseudomanifold. Assume that the restriction $f|_{f^{-1}(R)} : f^{-1}(R) \longrightarrow R$ is a fiber bundle with fiber $F$, a closed orientable topological manifold, and that $\operatorname{codim}_{X}(f^{-1}(R)) \geq 2$. Let $\overline{p} \in \{\overline{n}, \overline{m} \}$ be a perversity. Moreover,
        \begin{enumerate}
    	\item suppose that the normal bundle of $B:=f^{-1}(R)$ exists and is orientable, and that for each $r \in R$ there is a trivializing neighborhood $U$ of $r$ such that the connecting homomorphism in the Gysin sequence of the normal bundle of $B$ restricted to $f^{-1}(U)$ vanishes in degree $\operatorname{codim}_{X}(B)-1$.
    	\end{enumerate}
    	Let $R(f\vert_{X \setminus B})_{\ast}\underline{\mathbb{Q}}_{X \setminus B} \simeq \underline{\mathbb{Q}}_{Y \setminus R} \oplus \mathcal{L}$, where $\mathcal{L}$ is the local system on $Y \setminus R$ induced by the covering map $f\vert_{X \setminus B}$. If $\overline{p}(\operatorname{codim}_{Y}(R))+1 = \operatorname{codim}_{X}(B)$ and $\operatorname{codim}_{Y}(R)$ is even, then
    \begin{align}
    	Rf_{\ast}\underline{\mathbb{Q}}_{X}[n] \simeq h_{\ast} \big((R^{0}f\vert_{B})_{\ast} \underline{\mathbb{Q}}_{B}[\dim(B)]\big) \oplus IC^{\overline{p}}_{Y} (\underline{\mathbb{Q}}_{Y \setminus R} \oplus \mathcal{L}),
    \end{align}
    where $h: R\hookrightarrow Y$ is the inclusion, and $IC^{\overline{p}}_{Y}$ is Deligne's sheaf complex with respect to the stratification $Y \supset R$.
    \end{proposition*}    
    Note that if the spaces $X$, $B$, and $R$ are even-dimensional, then the condition $\overline{p}(\operatorname{codim}_{Y}(R))+1 = \operatorname{codim}_{X}(B)$ is equivalent to the dimension condition for semi-small morphisms, i.e. if the equality in Dimension Condition \ref{DimCond} holds.
    
    In the final step, we consider almost covering maps such that the decomposition of the fibered set $R= \bigsqcup_{\beta}T_{\beta}$, specified in the above definition, consists of more than one member. It turns out that, as in the algebraic case, only the relevant part of the fibered set $R$, namely the fiber bundles $f\vert_{f^{-1}(T_{\beta})}$ that satisfy the condition $2 \dim(f^{-1}(r))+\dim(T_{\beta})=\dim(X)$ for $r \in T_{\beta}$, contributes to the decomposition of $Rf_{\ast}\underline{\mathbb{Q}}_{X}$. To be more precise, we define topologically semi-small maps and the relevant part of the fibered set as follows.
    
       \begin{definition*}
    		Let $f:X \longrightarrow Y$ be an almost covering where $X$ is a closed $n$-manifold, $Y$ is a closed $n$-pseudomanifold. Let the filtration $Y \supset Y_{n_{R}}(=R) \supset Y_{n_{R}-1} \supset \cdots$ be the refined stratification on $Y$ obtained by Proposition~\ref{RefinedStra}, where $R$ is the fibered set and $n_{R}:= \dim(R)$. The almost covering $f:X \longrightarrow Y$ is called a \textbf{topologically semi-small map} if for each point $r \in S_{k}(:=Y_{k} \setminus Y_{k-1})$ and for each $0 \leq k \leq n_{R}$, the condition $2 \dim(f^{-1}(r))+\dim(S_{k}) \leq n$ is satisfied. Furthermore, we the set $R_{\operatorname{rel}}:=\{S_{k} \ \vert \  2 \dim(f^{-1}(r))+\dim(S_{k}) = n \ \forall \ r \in f^{-1}(S_{k})\}$ is referred to as the set of relevant strata.
    \end{definition*}

      The following theorem, which is the main result of this work, is a generalization of the previous proposition to topologically semi-small maps.
    
     \begin{theorem*}[Theorem \ref{DecoTheo3}]
        	Let $f : X \rightarrow Y$ be a topologically semi-small map such that $X$ is a closed topological manifolds, the fibered set $R$ is a closed stratified space, and $Y$ is a closed topological $n$-pseudomanifold. Set $B:=f^{-1}(R)$, let $R=\bigsqcup_{\beta}T_{\beta}$ be a suitable decomposition of the fibered set, and let $R_{\operatorname{rel}}$ be its relevant part. Furthermore, let the filtration $Y \supset Y_{n_{R}}(=R) \supset Y_{n_{R}-1} \supset \cdots$ be the refined stratification on $Y$ obtained by Proposition~\ref{RefinedStra}, where $n_{R}:= \dim(R)$. Set $S_{k}:= Y_{k} \setminus Y_{k-1}$, and assume the followings:
        \begin{enumerate}
        	\item for each $r \in R$ the fiber $f^{-1}(r)$ is a closed orientable topological manifold;
        	\item the normal bundle of the manifold $f^{-1}(S_{k})$ in $X$ exists, for each $0 \leq k \leq n_{R}$, and is orientable, and for each $S_{k'}\in R_{\operatorname{rel}}$ and each $r\in S_{k'}$ there is a small open neighborhood $U$ of $r$ such that the connecting homomorphism in the Gysin sequence of the normal bundle of $S_{k}$ restricted to $f^{-1}(U)$ vanishes in degree $\operatorname{codim}_{X}(B)-1$.
        	\end{enumerate}
        	Let $R(f|_{X \setminus B})_{*}\,\underline{\mathbb{Q}}_{X \setminus B}\simeq \underline{\mathbb{Q}}_{Y \setminus R}\oplus \mathcal{L}$, where $\mathcal{L}$ is the local system on $Y \setminus R$ associated to the covering map $f|_{X \setminus B}$. If each $\operatorname{codim}_{Y}(S_{k})$ is even and $\operatorname{codim}_{X}(B)\geq 2$, then
        \begin{align*}
        	Rf_{\ast}\underline{\mathbb{Q}}_{X}[n]\simeq \bigoplus_{S_{k} \in R_{\operatorname{rel}}}\iota_{k\ast}(IC_{\overline{S_{k}}} (R^{0}f\vert_{f^{-1}(S_{k})}\underline{\mathbb{Q}}_{S_{k}})) \oplus IC_{Y}(\underline{\mathbb{Q}}_{Y \setminus R} \oplus \mathcal{L}),
        \end{align*}
        where $\iota_{k}:\overline{S_{k}} \hookrightarrow Y$ is the canonical closed inclusion, and $IC^{\overline{p}}_{Y}$ is Deligne's intersection complex with respect to the above refined stratification and the perversity $\overline{p}\in\{\overline{n},\overline{m}\}$.
       \end{theorem*}
    
     Note that in the situation of the above theorem, if $B_{\operatorname{rel}} = \emptyset$, the quasi-isomorphism reduces to $Rf_{\ast}\underline{\mathbb{Q}}_{X} \simeq IC^{\overline{p}}_{Y} (\underline{\mathbb{Q}}_{Y \setminus R} \oplus \mathcal{L})$. This quasi-isomorphism can be considered as the topological analogous of the decomposition theorem for algebraic small morphisms.

     \section{Notation and Definitions}

   In this section, we introduce the definitions and notations needed for this work. We begin by defining the main object of study.
     
     \begin{definition}\label{AlMoCov}
     A continuous proper surjection $ f: X \longrightarrow Y$ between path-connected, locally compact, Hausdorff spaces $X$ and $Y$ is called an \textbf{almost covering map}
     if 
     \begin{enumerate}
     	\item there exists a nowhere closed dense set $R \subset Y$ such that the restriction $f \vert_{f^{-1}(Y \setminus R)} : f^{-1}(Y \setminus R) \longrightarrow
     	Y \setminus R$ is a finite covering map, and
     	\item there is a finite decomposition $R=\bigsqcup_{\beta}T_{\beta}$ such that the restriction $f \vert_{f^{-1}(T_{\beta})}:f^{-1}(T_{\beta}) \longrightarrow T_{\beta}$ is a fiber bundle for each $\beta$.
     \end{enumerate}
     The set $Y \setminus R$ is called the regular set of the almost covering map $f$, and the set $R$ is referred to as a fibered set of the map $f$.
     \end{definition}
     
   Since the fibered set $R$ is not uniquely determined by the above definition, the following remark is important. 

    \begin{remark}
	Observe that for a given almost covering $f: X \longrightarrow Y$, the fibered set $R\subset Y$ is not unique: any closed nowhere dense set $R'$ containing $R$ is also a fibered set. In this paper, by \emph{the fibered set $R$} we mean the minimal branch locus, i.e., the set of all points $y\in Y$ that admit no evenly covered neighborhood. Moreover, the decomposition $R=\bigsqcup_{\beta}T_{\beta}$ is not unique; in this work we always take the coarsest such decomposition, unless stated otherwise. 
    \end{remark}
     
     Similar to the inclusion of the branch locus of a branched covering into the ambient space, the inclusion of the fibered set into the ambient space $R \hookrightarrow Y$ can be arbitrarily wild. We generalize the definition of local flatness in what follows and consider only almost coverings for which the inclusions $R \hookrightarrow Y$ and $f^{-1}(R) \hookrightarrow X$ are locally flat.

    \begin{definition}\label{localflatness}
	Let $Y = \bigsqcup_{\eta} W_{\eta}$ and $R = \bigsqcup_{\alpha} S_{\alpha}$ be topologically stratified spaces such that $R \subseteq Y$. An inclusion $R \hookrightarrow Y$ is said to be \textbf{locally flat at a point $r \in R$} if the following holds. Suppose $r \in W_{\eta} \cap S_{\alpha}$ and $\dim(W_{\eta}) \leq \dim(S_{\alpha})$. Then there exist distinguished neighborhoods $U \subset R$ and $U^{\prime} \subset Y$ of $r$ such that $U \subset U^{\prime}$ and the topological pair $(U \cap S_{\alpha},\, U^{\prime} \cap W_{\eta})$ is homeomorphic to the topological pair $(\mathbb{R}^{\dim S_{\alpha}},\, \mathbb{R}^{\dim W_{\eta}})$ with the standard inclusion $\mathbb{R}^{\dim W_{\eta}} \hookrightarrow \mathbb{R}^{\dim S_{\alpha}}$. In other words, there exist homeomorphisms $U^{\prime} \cap W_{\eta} \longrightarrow \mathbb{R}^{\dim W_{\eta}}$ and $U \cap S_{\alpha} \longrightarrow \mathbb{R}^{\dim S_{\alpha}} $ making the following diagram commute:
	\begin{align*}		
		\begin{tikzcd}[ampersand replacement=\&]
			U^{\prime} \cap W_{\eta} \arrow[r,  hookrightarrow] \arrow[d] \& U \cap S_{\alpha} \arrow[d]   \\
			\mathbb{R}^{\dim W_{\eta}} \arrow[r, hookrightarrow]  \& \mathbb{R}^{\dim S_{\alpha} }  
		\end{tikzcd}.
	\end{align*}
	The inclusion $R \hookrightarrow Y$ is locally flat if it is locally flat at each point $r \in R$. If $\dim(W_{\eta}) \geq \dim(S_{\alpha})$, we define local flatness similarly.
    \end{definition}

      In this work, we consider almost coverings where the domain is a topological manifold and the target space is a topological pseudomanifold. Topological pseudomanifolds were first defined and studied by Goresky and MacPherson in \cite{GoreskyMacPherson1983}. In the following, we briefly give the necessary definitions for this work. 

    \begin{definition}\label{pseudomanifolds}
	We define a \textbf{topologically stratified space} inductively on dimension. A 0-dimensional topologically stratified space $X$ is a countable set with the discrete topology. For $m > 0$ an \textbf{$m$-dimensional topologically stratified space} is a paracompact Hausdorff topological space $X$ equipped with a filtration
	\begin{align*}
		X=X_{m} \supseteq X_{m-1} \supseteq \dots \supseteq X_{1} \supseteq X_{0} \supseteq X_{-1}= \emptyset
	\end{align*}
	by closed subsets $X_{j}$ such that if $x \in X_{j}-X_{j-1}$ there exists a neighborhood $\mathcal{N}_{x}$ of $x$ in $X$, a compact $(m-j-1)$-dimensional topologically stratified space $L$ with filtration
	\begin{align*}
		L=L_{m-j-1} \supseteq \dots \supseteq L_{1} \supseteq  L_{0} \supseteq L_{-1} = \emptyset,
	\end{align*}  
	and a homeomorphism $\phi : \mathcal{N}_{x} \longrightarrow \mathbb{R}^{j} \times \mathcal{C}(L),$
	where $\mathcal{C}(L)$ is the open cone on $\mathcal{L}$, such that $\phi$ takes $\mathcal{N}_{x} \cap X_{j+i+1}$ homeomorphically onto
	\begin{align*}
		\mathbb{R}^{j} \times \mathcal{C}(L_{i}) \subseteq \mathbb{R}^{j} \times \mathcal{C}(L)
	\end{align*}
	for $m-j-1 \geq i \geq 0$, and $\phi$ takes $\mathcal{N}_{x} \cap X_{j}$ homeomorphically onto
	\begin{align*}
		\mathbb{R}^{j} \times \{ \text{vertex of }\; \mathcal{C}(L) \}.
	\end{align*}
\end{definition}
    
   \begin{remark}\label{linkandstratum}
   	It follows that $X_{j}-X_{j-1}$ is a $j$-dimensional topological manifold. (The empty set is a manifold of any dimension.) We call the connected components of these manifolds the \textbf{strata} of $X$. Any $L$ that satisfies the above properties is referred to as a \textbf{link} of the stratum at $x$.   
   \end{remark}
   In numerous applications, the filtration
   $X_m \supset X_{m-1} \supset \cdots \supset X_0 \supset X_{-1} = \emptyset$
   is less amenable to direct manipulation than its associated stratified decomposition. Specifically, partitioning the space into the disjoint union of strata, $X = \bigsqcup_{j=0}^{m} \left( X_j \setminus X_{j-1} \right),$
   often provides a considerably more tractable framework. This decomposition effectively isolates the incremental layers of the filtration, thereby facilitating clearer combinatorial and topological reasoning---particularly when explicit constructions, cohomological computations, or inductive arguments over the strata are required---as opposed to working with the original filtration, which inherently involves overlapping relational dependencies. The following remark follows immediately from the definitions.
   
   \begin{remark}
   	A decomposition $X = \bigsqcup_{i} S_i$ yields a \textbf{topological stratification} of $X$ if the following conditions hold:
   	
   	\begin{enumerate}
   		\item \textbf{Local finiteness.}
   		The family $\{S_i\}$ is locally finite: every $x \in X$ has a neighbourhood that intersects only finitely many strata.
   		
   		\item \textbf{Manifold strata.}
   		Each $S_i$ is locally closed in $X$ and, with its subspace topology, is a topological manifold of constant dimension $d_i$.
   		
   		\item \textbf{Frontier condition.}
   		If $S_i \cap \overline{S_j} \neq \varnothing$ with $i \neq j$, then $S_i \subset \overline{S_j}$ and $\dim S_i < \dim S_j$.
   		Equivalently, the closure of any stratum is a union of strata.
   		
   		\item \textbf{Local conical structure (topological regularity).}
   		For every $x \in S_i$ there exists an open neighborhood $U$ of $x$ in $X$, a topological stratified space $\mathcal{L}$, a neighborhood $V$ of $x$ in $S_i$, and a stratum-preserving homeomorphism $\phi : U \longrightarrow V \times C(L)$
   		The product and the cone carry the natural stratifications induced by the stratification of $L$.
   	\end{enumerate}
   \end{remark}

    \begin{remark}
    	The \emph{intrinsic stratification} of a stratified space $X$ is defined as follows: one begins by taking the top stratum to be the maximal open submanifold of $X$ and then inductively defines the lower strata according to dimension. For details, we refer the reader to \cite[V]{borel}.
    \end{remark}

    \begin{definition}\label{PSMFD}
    	An \textbf{$m$-dimensional topological pseudomanifold} is a para-compact Hausdorff topological space $X$ which possesses a topological stratification such that $X_{m-1}=X_{m-2}$
    	and $X-X_{m-1}$ is dense in $X$.
    \end{definition}
   Following Goresky and MacPherson \cite{GoreskyMacPherson1983}, we now introduce the perversity parameter, which controls the allowable degree of intersection of chains with the singular strata of a stratified pseudomanifold.
   \begin{definition}
   	A \textbf{perversity} 
   	$\bar{p}: \mathbb{Z}_{ \geq 2} \longrightarrow \mathbb{Z}$ is a function such that $\overline{p}(2)=0$ and
   	$\overline{p}(k+1)- \overline{p}(k) \in \{1,0\}$. The \textbf{complementary perversity} $\bar{q}$ of $\bar{p}$ 
   	is the one with $\overline{p}(k)+\overline{q}(k)=k-2$. Furthermore, the lower middle perversity is defined by $\bar{m}(k)=[\frac{k-2}{2}]$. Its complementary perversity is called the upper middle perversity and is defined by $\bar{n}(k)=[\frac{k-1}{2}]$. 
   \end{definition}

   We now recall the recursive construction of Deligne's sheaf complex, defined stratum by stratum with respect to the perversity function $\overline{p}$. This complex is a central object of study in the present paper.

   \begin{definition}\label{DeSheaf}
   	Let $X^{n}$ be an $n$-dimensional pseudomanifold with a filtration
   	$X = X_{n} \supseteq X_{n-2} \supseteq \cdots \supseteq X_{0} \supseteq \emptyset$, and set
   	$U_{k} = X - X_{n-k}$ for $k \geq 2$, with inclusion maps $i_{k}:U_{k} \hookrightarrow U_{k+1}$. Let $\mathcal{L}$ be a local system on $X-X_{n-2}$. Deligne's sheaf complex $IC^{\overline{p}}_{X}(\mathcal{L})$ with respect to the perversity $\overline{p}$ is defined as 
   	\begin{align*}
   		IC^{\overline{p}}_{X}(\mathcal{L}) = \tau_{\leq \overline{p}(n)-n} Ri_{n \ast}  \cdots \tau_{\leq \overline{p}(3)-n} Ri_{3\ast}   \tau_{\leq \overline{p}(2)-n} Ri_{2 \ast} \mathcal{L}[n].
   	\end{align*} 
   \end{definition}
   
   Let $X$ and $Y$ be topologically stratified spaces. We now recall the definition of a stratified map between them.  
   
    \begin{definition}
   	A continuous map $f: X \longrightarrow Y$ is \emph{stratified} if it satisfies the following two conditions:
   	\begin{enumerate}
   		\item For any connected component $S$ of any stratum $Y_i \setminus Y_{i-1}$, the set $f^{-1}(S)$ is a union of connected components of strata of $X$.
   		\item For each point $p \in Y_i \setminus Y_{i-1}$ there exists a neighborhood $N$ of $p$ in $Y_i$, a topologically stratified space $F = F_k \supset F_{k-1} \supset \cdots \supset F_{-1} = \emptyset$,
   		and a stratum-preserving homeomorphism $F \times N \to f^{-1}(N)$ which commutes with the projection to $N$.
   	\end{enumerate}
   	
   \end{definition}
   
   As shown in \cite{Rourke1967}, a stratum in a topological pseudomanifold (Definition \ref{PSMFD}) does not, in general, possess a tubular neighborhood in the ambient space. However, for the purposes of this work, we require that the strata of a given stratification admit tubular neighborhoods in the ambient stratified space. The existence of such neighborhoods is guaranteed in topological stratified spaces with additional structure, such as Whitney stratified spaces \cite{Mather2012} or Thom--Mather stratified spaces \cite{mather1973stratifications}. However, it turns out that the additional structure required for these spaces is unnecessary for the present work. Accordingly, we restrict our attention to stratified spaces in a purely topological setting, assuming that the strata admit tubular neighborhoods. In particular, we require that the strata of the preimage of the fibered set admit normal bundles in the domain of a given almost covering.

     Let $f: X \longrightarrow Y$ be an almost covering. We will show that, the Gysin sequence of the normal bundle of $f^{-1}(R)$ in $X$, if it exists, determines whether the complex of sheaves $Rf_{\ast} \underline{\mathbb{Q}}_{X}$ admits a decomposition. To be more precise, we are interested in the connecting homomorphisms in the Gysin sequence. In order to avoid any ambiguity, we define the degree of a connecting homomorphism of the Gysin sequence as follows.

        \begin{definition}
        	Let $ \pi : E \longrightarrow B $ be a fiber bundle with fiber $S^{k}$ ($ k \geq 1 $), over a connected base $B$. 
        	Assume for simplicity that the bundle is oriented, so that the fibers carry a consistent orientation. In the Gysin sequence 
        	\begin{align*}
        	\cdots \longrightarrow H^{i}(B) \rightarrow H^{i}(E) \xrightarrow{ \ \delta_{i}  \ } H^{i-k}(B) \rightarrow H^{i+1}(B) \longrightarrow \cdots,
        	\end{align*}
        	we call $\delta_{i}:H^{i}(E) \longrightarrow H^{i-k}(B)$ \textbf{the connecting homomorphism in degree $i$}.
        \end{definition}

       \section{Topology of Almost Covering Maps}\label{TopAlCov}

       In this section, we study the topology of almost covering maps. Although our applications require only special cases, several of the results are established in greater generality.

        Let $R=\bigsqcup_{\beta}T_{\beta}$ be a decomposition of $R$ into topologically stratified spaces $T_{\beta}$. Assume that each $T_{\beta}$ is endowed with a stratification $T_{\beta}= \bigsqcup_{\gamma}V_{\gamma}^{\beta}$. In general, the decomposition $R=\bigsqcup_{\beta,\gamma}V_{\gamma}^{\beta}$ does not yield a stratification of $R$. The following lemma gives sufficient conditions for this decomposition to be a stratification of $R$.

          \begin{lemma}\label{StratR}
       	Let $R$ be a closed topologically stratified space with the intrinsic stratification $R=\bigsqcup_{\alpha} S_{\alpha}$. Consider a finite decomposition $R=\bigsqcup_{\beta} T_{\beta}$ such that each $T_{\beta}$ is a topologically stratified space with a stratification $T_{\beta}=\sqcup_{\gamma} V^{\beta}_{\gamma}$. Assume the following:
    	\begin{enumerate}
		\item The inclusion $\overline{V^{\beta}_{\gamma}} \hookrightarrow R$ is locally flat, for all $\beta$ and $\gamma$.
		\item For each point $r \in R$ there is an open neighborhood that intersects finitely many connected components of finitely many $V^{\beta}_{\gamma}$. 
		\item The collection $\{ V_{\gamma}^{\beta}\}$ satisfies the frontier condition, i.e., $V^{\beta}_{\gamma} \cap \overline{V^{\beta^{\prime}}_{\gamma^{\prime}}} \neq \emptyset$ implies $V^{\beta}_{\gamma} \subseteq \overline{V^{\beta^{\prime}}_{\gamma^{\prime}}}$ for all $\gamma, \gamma^{\prime}, \beta,$ and $\beta^{\prime}$. 
	   \end{enumerate}
	   Then the decomposition $R=\bigsqcup_{\alpha,\beta,\gamma} (V_{\gamma}^{\beta} \cap S_{\alpha})$ is a topological stratification.  
       \end{lemma}

       \begin{proof}
    	We verify the four axioms of a topological stratification (frontier condition, manifold strata, local finiteness, and the local cone condition).
	
    	\noindent 
    	Each $V^{\beta}_{\gamma}$ is a stratum of the topologically stratified space $T_\beta$, hence a topological manifold with the subspace topology from $T_\beta$, which coincides with the subspace topology from $R$. Let $r \in V_{\gamma}^{\beta} \cap S_{\alpha}$. If $\dim(V_{\gamma}^{\beta}) \leq \dim(S_{\alpha})$, then from the local flatness assumption it follows that there exists and open neighborhood $U$ of $r$ in $R$ such that $U \cap (S_{\alpha} \cap V_{\gamma}^{\beta}) \cong \mathbb{R}^{\dim(V_{\gamma}^{\beta})}$. Otherwise, if $\dim(V_{\gamma}^{\beta}) \geq \dim(S_{\alpha})$ we arrive at $U \cap (S_{\alpha} \cap V_{\gamma}^{\beta}) \cong \mathbb{R}^{\dim(S_{\alpha})}$. Therefore every element of the decomposition $R=\bigsqcup_{\alpha,\beta,\gamma} (V_{\gamma}^{\beta}\cap S_{\alpha})$ is a topological manifold. Furthermore, the local finiteness and the frontier condition are satisfied by assumption.
	
    	\noindent
    	It remains to construct for every $r\in R$ a stratified homeomorphism from an open neighborhood of $r$ to a cone over a link, compatible with the strata $V^\beta_\gamma$. 
    	Let $r\in V^{\beta_0}_{\gamma_0}$ and denote by $S_{\alpha_0}$ the stratum of the original stratification $R=\bigsqcup_{\alpha}S_{\alpha}$ containing $r$.

    	\noindent
    	Assume $\dim( V^{\beta_0}_{\gamma_0}) \leq \dim(S_{\alpha_{0}})$. 
    	Since $R$ is topologically stratified, there exists an open neighborhood $U$ of $r$ in $R$ and a stratified homeomorphism $\phi: U \xrightarrow{\cong} \mathbb{R}^{n_{\alpha_0}} \times C(L)$, where $C(L)$ is the cone on the compact topologically stratified space $L$, $n_{\alpha_0 } = \dim S_{\alpha_0}$, and $\phi(r) = (0,\text{cone point})$. The homeomorphism sends $U \cap S_{\alpha_0}$ to $\mathbb{R}^{n_{\alpha_0 }} \times \{\text{cone point}\}$, and the stratification of $U$ is the product of the trivial stratification on $\mathbb{R}^{n_{\alpha_0}}$ and the stratification of the cone over of $L$.
	
    	\noindent
    	Since $S_{\alpha_0} \cap V^{\beta_0}_{\gamma_0} \hookrightarrow S_{\alpha_0}$ is locally flat, after applying a stratum‑preserving homeomorphism of $\mathbb{R}^{n_{\alpha_0}}$ we may assume that $\phi\bigl(U\cap V^{\beta_0}_{\gamma_0}\bigr)= \mathbb{R}^{n_{\gamma_0}^{\beta_0}} \times \{0\} \times \{\text{cone point}\}$
    	inside $\mathbb{R}^{n_{\gamma_0}^{\beta_0}} \times \mathbb{R}^{n_{\alpha_0}-n_{\gamma_0}^{\beta_0}} \times C(L)$.
    	Since $T_{\beta_0}$ is a topologically stratified space, there exists a stratified homeomorphism $\psi : U\cap T_{\beta_0} \xrightarrow{\cong} \mathbb{R}^{n_{\gamma_0}^{\beta_0}} \times C(L')$ with $\psi(r)=(0,\text{cone point})$ and $\psi\bigl(U\cap V^{\beta_0}_{\gamma_0}\bigr) = \mathbb{R}^{n_{\gamma_0}^{\beta_0}}\times\{\text{cone point}\}$, where $L'$ is a compact stratified space (a link of $V^{\beta_0}_{\gamma_0}$ in $T_{\beta_0}$). Using $\phi$, $\psi$, and the inclusion $\iota:U\cap T_{\beta_0}\hookrightarrow U$, we obtain a stratum‑preserving, locally flat inclusion 
    	\begin{align*}
		f = \phi\circ\iota\circ\psi^{-1} : \mathbb{R}^{n_{\gamma_0}^{\beta_0}} \times C(L') \hookrightarrow \mathbb{R}^{n_{\gamma_0}^{\beta_0}} \times \mathbb{R}^{n_{\alpha_0}-n_{\gamma_0}^{\beta_0}} \times C(L)(\cong \mathbb{R}^{n_{\gamma_{0}}^{\beta_{0}} } \times C(S^{n_{\alpha_{0}}-n^{\beta_{0}}_{\gamma_{0}}-1 } \ast L)),
    	\end{align*}
    	which is the identity on $\mathbb{R}^{n_{\gamma_0}^{\beta_0}}$. Thus we obtain the commutative diagram
    	\begin{align*}
		\begin{tikzcd}[ampersand replacement=\&]
			U \cap (V^{\beta_{0}}_{\gamma_{0}} \cap S_{\alpha_{0}})(\cong \mathbb{R}^{n_{\gamma_{0}}^{\beta_{0}}}) \arrow[r, hookrightarrow] \arrow[d, hookrightarrow] \& U \cap T_{\beta_{0}} (\cong \mathbb{R}^{n_{\gamma_{0}}^{\beta_{0}}} \times C(L')) \arrow[d, hookrightarrow] \\
			U \cap S_{\alpha_0} \bigl(\cong \mathbb{R}^{n_{\gamma_{0}}^{\beta_{0}}} \times C(S^{(n_{\alpha_0}-n_{\gamma_{0}}^{\beta_{0}}-1)})\bigr) \arrow[r, hookrightarrow] \& U \bigl(\cong \mathbb{R}^{n_{\gamma_{0}}^{\beta_{0}}} \times C(S^{(n_{\alpha_0}-n_{\gamma_{0}}^{\beta_{0}}-1)} \ast L) \bigr)
		\end{tikzcd}.
    	\end{align*}
    	Note that all maps are stratum-preserving. In particular, the inclusion $	U \cap V^{\beta_{0}}_{\gamma_{0}} (\cong \mathbb{R}^{n_{\gamma_{0}}^{\beta_{0}}}) \hookrightarrow U \bigl(\cong \mathbb{R}^{n_{\gamma_{0}}^{\beta_{0}}} \times C(S^{(n_{\alpha_0}-n_{\gamma_{0}}^{\beta_{0}}-1)} \ast L) \bigr)$ maps $\mathbb{R}^{n_{\gamma_{0}}^{\beta_{0}}}$ to $\mathbb{R}^{n_{\gamma_{0}}^{\beta_{0}}} \times \{\text{cone point}\}$. Hence, the open neighborhood $U( \cong \mathbb{R}^{n_{\gamma_{0}}^{\beta_{0}}} \times C(S^{(n_{\alpha_0}-n_{\gamma_{0}}^{\beta_{0}}-1)} \ast L))$ provides a distinguished neighborhood of the point $r \in  V^{\beta_{0}}_{\gamma_{0}} \cap S_{\alpha_{0}}$ in $R$. Otherwise, if $\dim( V^{\beta_0}_{\gamma_0}) \geq \dim(S_{\alpha_{0}})$ a similar reasoning shows a distinguished neighborhood of $r$ in $R$ provides the local cone condition. The claim follows.
        \end{proof}

        In this work, we only consider the decompositions of $R$ that induce a stratification on $R$ in the sense of the above lemma. This motivates the following definition.

       \begin{definition}\label{SutStrat}
    	Let $R$ be a closed topologically stratified space with an intrinsic stratification $R=\bigsqcup_{\alpha} S_{\alpha}$. Consider a decomposition $R=\bigsqcup_{\beta} T_{\beta}$ such that each $T_{\beta}$ is a topologically stratified space with a stratification $T_{\beta}=\bigsqcup_{\gamma} V^{\beta}_{\gamma}$, and the conditions in Lemma \ref{StratR} are satisfied. We call the decomposition $R=\bigsqcup_{\beta} T_{\beta}$ a \textbf{suitable} decomposition. The stratification $R=\bigsqcup_{\alpha, \beta,\gamma}( V^{\beta}_{\gamma} \cap S_{\alpha})$ is referred to as the stratification induced by the suitable decomposition.
    \end{definition}

    Given an almost covering map $f:X \longrightarrow Y$ with a suitable decomposition of the fibered set, namely $R = \bigsqcup_{\beta}T_{\beta}$, and assuming that $R \subset Y$ is a closed topologically stratified space of codimension 2 and $Y$ is a closed topological pseudomanifold, one can use the stratification on $R$ induced by the suitable decomposition to refine the intrinsic stratification on $Y$.

    \begin{proposition}\label{RefinedStra}
	Let $f:X \longrightarrow Y$ be an almost covering map, where $X$ is a closed topological manifold and $Y$ is a closed topological pseudomanifold. Suppose the fibered set $R = \bigsqcup_{\beta} T_{\beta} \subset Y$ is equipped with a suitable decomposition. Assume that $R$ is a closed topologically stratified space of codimension greater than or equal to $2$. Let $R = \bigsqcup_{\alpha} S_{\alpha}$, $T_{\beta} = \bigsqcup_{\gamma} V^{\beta}_{\gamma}\ \text{for each } \beta$, and 
	$Y = \bigsqcup_{\eta} W_{\eta}$
	be intrinsic stratifications. Furthermore, assume that the inclusion $R \xhookrightarrow[]{\phantom{.} i \phantom{.} } Y$ is locally flat. Then the decomposition
	\begin{align}
	Y = \bigl( W_{\eta_{\text{top}}} \setminus (R \cap W_{\eta_{\text{top}}}) \bigr) \;\bigsqcup_{\alpha,\eta,\beta,\gamma}\; \bigl( W_{\eta} \cap V^{\beta}_{\gamma} \cap S_{\alpha}\bigr)
	\end{align}
	is a refined stratification on $Y$, where $W_{\eta_{\text{top}}}$ denotes the top stratum of $Y$.
    \end{proposition}

    \begin{proof}
    Let $Y_{\text{sing}} \subset Y$ denote the singular part of $Y$, and let $n=\dim(X)=\dim(Y)$. Assume that $Y_{\text{sing}} \setminus R \neq \emptyset$. Since $R \subset Y$ is a closed subspace, it follows that $(Y \setminus R)$ is open in $Y$. Given $y \in Y_{\text{sing}} \setminus R$ and $x \in X$ such that $f(x)=y$, there exists an open neighborhood $U \cong \mathbb{R}^{\dim(X)}$ of $x$ such that $f \vert_{U}$ is a homeomorphism. Hence $f(U) \cong \mathbb{R}^{\dim(Y)}$ is an open neighborhood of $y \in Y_{\text{sing}}$, which contradicts the fact that $y$ is a singular point. It follows that $Y_{\text{sing}} \subset R$.
    
    \noindent
    By Lemma \ref{StratR}, it follows that the decomposition $R=\bigsqcup_{\alpha,\beta,\gamma}(V^{\beta}_{\gamma}\cap S_{\alpha})$ is a stratification. Choose a distinguished open neighborhood of $r$ in $Y$, denoted by $U \cong \mathbb{R}^{n_{\eta}} \times C(L)$, where $n_{\eta}=\dim(W_{\eta})$, $L$ is a link of $r$ in $Y$, and the homeomorphism is stratum-preserving. Let $l=\dim(V^{\beta}_{\gamma} \cap S_{\alpha})$ and assume that $l \leq n_{\eta}$. The local flatness of the inclusion $R \xhookrightarrow[]{\phantom{.} i \phantom{.} } Y$ implies the existence of the following commutative diagram
    \begin{align*}
    	\begin{tikzcd}[ampersand replacement=\&]
    	U \cap V^{\beta}_{\gamma}  \arrow[r, hookrightarrow] \arrow[d,"\cong"] \& U \cap W_{\eta}  \arrow[d,"\cong"] \\
     \mathbb{R}^{l} \arrow[r, hookrightarrow] \& \mathbb{R}^{n_{\eta}}
    \end{tikzcd},
    \end{align*}
    such that the inclusion $\mathbb{R}^{l} \hookrightarrow \mathbb{R}^{n_{\eta}}$ is homeomorphic to the standard inclusion. We can rewrite the previous inclusion as $\mathbb{R}^{l } \hookrightarrow \mathbb{R}^{l} \times C(S^{n_{\eta}-l-1 })$, where $\mathbb{R}^{n^{\beta}_{\gamma} }$ is mapped to $\mathbb{R}^{n^{\beta}_{\gamma} } \times \{\text{cone point} \}$.

    \noindent
    Since $R$ is a topologically stratified space, for $U$ chosen sufficiently small, we have a stratum-preserving homeomorphism $U \cap R \cong \mathbb{R}^{l} \times C(L')$, where $L'$ is a link of the point $r$ in $R$. Note that the composition of  the homeomorphism $U \cong \mathbb{R}^{n_{\eta}} \times C(L)$, the inclusion  $R \xhookrightarrow[]{\phantom{.} i \phantom{.} } Y$, and a stratum-preserving inverse homeomorphism of $U \cap R \cong \mathbb{R}^{l} \times C(L')$ yields a stratum-preserving inclusion $\mathbb{R}^{l} \times C(L') \hookrightarrow \mathbb{R}^{l} \times C(S^{n_{\eta}-l-1}\ast L)$, that maps $\mathbb{R}^{l} \times \{\text{cone point}\}$ to $\mathbb{R}^{l} \times \{\text{cone point}\}$, where $S^{n_{\eta}-l-1}\ast L$ is the topological join of the spaces $S^{n_{\eta}-l-1}$ and $L$. Hence, we arrive at the following commutative diagram 
   \begin{align*}
   	\begin{tikzcd}[ampersand replacement=\&]
   		U \cap (V^{\beta}_{\gamma}\cap S_{\alpha} \cap W_{\eta})(\cong 	\mathbb{R}^{l})  \arrow[r, hookrightarrow] \arrow[d,hookrightarrow] \& U \cap W_{\eta}\big( \cong \mathbb{R}^{l} \times C(S^{n_{\eta}-l-1})  \big)  \arrow[d,hookrightarrow] \\ U \cap R (\cong \mathbb{R}^{l} \times C(L'))
   	 \arrow[r, hookrightarrow] \& U \big(\cong \mathbb{R}^{l} \times C(S^{n_{\eta}-l-1}\ast L)\big)
   	\end{tikzcd},
   \end{align*}
    in which all inclusions are stratum-preserving. Hence, the inclusion $	U \cap (V^{\beta}_{\gamma} \cap S_{\alpha} \cap W_{\eta})(\cong 	\mathbb{R}^{l})  \hookrightarrow  U \big(\cong \mathbb{R}^{l} \times C(S^{n_{\eta}-l-1}\ast L)\big)$ maps $\mathbb{R}^{l}$ to $\mathbb{R}^{l} \times \{\text{cone point}\}$. As a result, $U(\cong \mathbb{R}^{l} \times C(S^{n_{\eta}-l-1}\ast L))$ provides a distinguished neighborhood of the point $r$ in $Y$.
    
    \noindent
    If $n_{\eta} \leq l$, we use the same arguments as before and arrive at the following commutative diagram
      \begin{align*}
    	\begin{tikzcd}[ampersand replacement=\&]
    	U \cap (V^{\beta}_{\gamma} \cap S_{\alpha} \cap W_{\eta})(\cong 	\mathbb{R}^{n_{\eta}})  \arrow[r,"\cong"] \arrow[d,hookrightarrow] \& U \cap W_{\eta} ( \cong \mathbb{R}^{n_{\eta}})  \arrow[d,hookrightarrow] \\ U \cap R (\cong \mathbb{R}^{l} \times C(L'))
    	\arrow[r, hookrightarrow] \& U \big(\cong \mathbb{R}^{l} \times C(L)\big)
    \end{tikzcd},
    \end{align*}
    in which all inclusions are stratum-preserving. Consequently, the inclusion $	U \cap (V^{\beta}_{\gamma} \cap S_{\alpha} \cap W_{\eta})(\cong 	\mathbb{R}^{n_{\eta}})  \hookrightarrow  U \big(\cong \mathbb{R}^{n_{\eta}} \times C( L)\big)$, maps $\mathbb{R}^{n_{\eta} }$ to $\mathbb{R}^{n_{\eta} } \times \{\text{cone point}\}$. As a result, $U$ provides a distinguished neighborhood of the point $r$ in $Y$. The claim follows.
    \end{proof}

    For the purposes of this work, we still need to show that in the situation of the preceding proposition, the pull-back of the refined stratification on $Y$ under the almost covering map $f$ induces a stratification on $X$. Note that the total space of a fiber bundle over a topologically stratified space with arbitrary fibers need not itself be stratified in general. The following elementary lemma shows that the total space of a fiber bundle over a stratified space with stratified fibers admits a stratification.
    \begin{lemma}\label{StratTotSpac}
   	Let $f:E \longrightarrow T$ be a fiber bundle whose fiber is a stratified space $F$, such that the local trivializations are stratum-preserving homeomorphisms. If $T$ is a stratified space, then $E$ admits a stratification under which the map $f$ is stratified.
    \end{lemma}
    \begin{proof}
    	Consider the intrinsic stratification on $F$. Then any homeomorphism $F \longrightarrow F$ is stratum-preserving. Note that this is not true for an arbitrary choice of stratification on $F$.

    	\noindent
    	Choose a trivializing cover $\{U_{\sigma}\}$ of $T$. Hence, for each pair $U_{\alpha},U_{\beta} \in \{U_{\sigma}\}$ and each $t \in U_{\alpha} \cap U_{\beta}$ the transition maps $g_{\alpha \beta}:F \longrightarrow F$ preserve the intrinsic stratification on $F$.
    	
    	\noindent
    	Let $V_{\gamma}$ be a stratum of $T$. The restriction of $f$ to $E \vert_{V_{\gamma}}:=f^{-1}(V_{\gamma})$, is a fiber bundle over the manifold $V_{\gamma}$ with fibers homeomorphic to the stratified space $F$. The collection $\{U_{\sigma} \cap V_{\gamma}\}$ forms an open cover of $V_{\gamma}$, and the restrictions of the local trivializations give a trivializing cover for $E \vert_{V_{\gamma}}$. Consider the stratification on $(U_{\sigma} \cap V_{\gamma}) \times F $ which is induced by the product of the intrinsic stratification on the fiber $F$ and the trivial stratification on the manifold $U_{\sigma} \cap V_{\gamma}$. Since the transition maps are stratum-preserving, gluing the local charts $(U_{\sigma} \cap V_{\gamma}) \times F $ using the transition maps yields a stratification on $E\vert_{V_{\gamma}}$.
    	
    	\noindent
    	Let $e \in E \vert_{V_{\gamma}}$ and set $t=f(e)$. Since $T$ is a stratified space, there exists a distinguished neighborhood $U_{T} \subset T$ of $t$ admitting a stratum-preserving homeomorphism $U_{T} \cong \mathbb{R}^{k} \times C(L_{T})$, where the stratified space $L_{T}$ is a link of the point $t$ in $T$. Moreover, shrinking $U_{T}$ if necessary, we have $f^{-1}(U_{T}) \cong U_{T} \times F$. The topological space $F$ is stratified and hence there is a distinguished neighborhood $U_{F} \cong \mathbb{R}^{l} \times C(L_{F})$, where $L_{F}$ is a stratified space and the homeomorphism is stratum-preserving, so that $U_{T} \times U_{F}$ is an open neighborhood of $e$. Note that there is a stratum-preserving homeomorphism $C(L_{T}) \times C(L_{F}) \cong C(L_{T} \ast L_{F})$, where $L_{T} \ast L_{F}$ denotes the join of $L_{T}$ and $L_{F}$. Hence $U_{T} \times U_{F} \cong \mathbb{R}^{k+l} \times C(L_{T} \ast L_{F})$, thereby providing a distinguished neighborhood of $e$ in $E$. Moreover, this homeomorphism is stratum-preserving. The claim follows.
    \end{proof}
    
    Note that the stratification which we obtain by the above lemma is not, in general, the intrinsic stratification of the space $E$.

    Let $f:X \longrightarrow Y$ be an almost covering with the fibered set $R$. Assume the setup of Lemma \ref{StratR} and consider the stratification $R=\bigsqcup_{\alpha,\beta,\gamma} (V_{\gamma}^{\beta} \cap S_{\alpha})$. If the hypotheses of the preceding lemma are satisfied for the fiber bundle $f \vert_{f^{-1}(V_{\gamma}^{\beta} \cap S_{\alpha})}$, then the topological space $f^{-1}(V_{\gamma}^{\beta} \cap S_{\alpha})$ admits a stratification. Now, assume that $(V_{\gamma}^{\beta} \cap S_{\alpha}) \cap \overline{(V_{\gamma'}^{\beta'} \cap S_{\alpha'})} \neq \emptyset$. Note that the fibers $f^{-1}(r)$ and $f^{-1}(r')$, where $r \in V_{\gamma}^{\beta} \cap S_{\alpha}$ and $r' \in V_{\gamma'}^{\beta'} \cap S_{\alpha'}$, are in general not homeomorphic. We still need to address the question of under which conditions the topological space $f^{-1}(\overline{V_{\gamma'}^{\beta'} \cap S_{\alpha'}})$ can be stratified.

    To this end, we first consider the following simplified setting. Let $f:X \longrightarrow [0,1]$ be a proper surjection such that $f^{-1}((0,1]) \cong F_{1} \times (0,1]$, where $F_{1}$ is a stratified space. Furthermore, we assume that the topological space $f^{-1}(\{0\})$ is stratified. As the following example shows, there does not in general exist a canonical map $g$ from $F_{1}$ to $F_{0}$ such that $X$ is homeomorphic to the mapping cylinder $\operatorname{Cyl}(g)$. In fact, merely requiring that the fibers $F_{0}$ and $F_{1}$ are topologically stratified spaces does not determine the topology of the space $f^{-1}([0,\epsilon])$ for arbitrarily small $\epsilon$, as the following example shows. However, within these more refined frameworks for stratified spaces, the existence of refined stratifications with respect to which the map $f$ is stratified follows automatically from well-known results and therefore does not require a separate proof.

    \begin{example}
    	Let \(X \subset \mathbb{R}^3\) be given in cylindrical coordinates by
    	\begin{align*}
    	X = \bigl\{ (\cos(1/t),\, \sin(1/t),\, t) \mid t \in (0,1) \bigr\}
    	\cup \bigl\{ (\cos\theta,\, \sin\theta,\, 0) \mid \theta \in [0,2\pi) \bigr\}.
    	\end{align*}
    	Define \(f \colon X \to [0,1)\) by \(f(r,\theta,z) = z\) (i.e., projection onto the \(z\)-coordinate). Then
    	\begin{itemize}
    		\item \(f^{-1}(0) = S^1\) (the unit circle), which is a compact 1‑manifold, hence a topologically stratified space.
    		\item $f^{-1}\bigl((0,1)\bigr) = \bigl\{ (\cos(1/t),\sin(1/t),t) \mid t \in (0,1) \bigr\}$ is homeomorphic to $\{\text{pt}\} \times (0,1)$. Thus we may take $F_1 = \{\text{pt}\}$ (a single point), a $0‑$manifold and therefore stratified.
    		\item \(f\) is continuous and proper: the preimage of any compact interval $[0,\varepsilon]$ is the union of the circle and the tail $\{ t \in (0,\varepsilon] \}$ of the spiral; this union is compact.
    	\end{itemize}
    	As \(t \to 0^+\), the point \((\cos(1/t),\sin(1/t),t)\) accumulates on the \emph{entire} circle \(S^1\). Consequently, there is no well‑defined (continuous) map \(g \colon \{\text{pt}\} \to S^1\) that sends the unique point of \(F_1\) to its ``limit'' in \(F_0\). Hence, from the mere assumptions of properness and stratified fibers, one cannot deduce the existence of an induced map $g$. Note also that the space $X$ in this example does not admit a topological stratification.
    \end{example}

    Hence, in the situation of the above example, we need a minimal condition on the space $X$ that controls the topology of the preimage $f^{-1}([0,\epsilon])$ for arbitrarily small $\epsilon$.
    
    \begin{lemma}\label{FibersMap}
    Let $f:X \longrightarrow [0,1]$ be a proper surjection such that the restriction $f\vert_{f^{-1}((0,1])}:f^{-1}((0,1]) \longrightarrow (0,1]$ is a trivial fiber bundle with closed stratified fibers; i.e., $f^{-1}((0,1]) \cong F_{1}\times (0,1]$,
    where $F_{1}$ is a closed stratified space. Suppose that the fiber $F_{0}:=f^{-1}(\{0\})$
    is a closed stratified space. If $X$ is a topologically stratified space, then there exists a map $g:F_{1}\longrightarrow F_{0}$ such that $X$ is homeomorphic to the mapping cylinder $\operatorname{Cyl}(g)$.
    \end{lemma}

    \begin{proof}
    	\newcommand{\cyl}{\operatorname{Cyl}}
    	By hypothesis there exists a homeomorphism
    		$\varphi\colon F_1 \times (0,1] \xrightarrow{\cong} f^{-1}\bigl((0,1]\bigr)$
    		that respects the projection onto $(0,1]$, i.e.\
    		$f(\varphi(x,t)) = t$ for all $(x,t) \in F_1 \times (0,1]$.
    		Moreover, $F_0$ and $F_1$ are closed, hence compact; the properness of $f$ and
    		compactness of $[0,1]$ imply that $X = f^{-1}([0,1])$ is compact as well.
    		
    		\noindent
    		Fix $x \in F_1$ and consider the curve $\gamma_x(t) = \varphi(x,t)$ for
    		$t \in (0,1]$. Because $X$ is compact, the set $\{\gamma_x(t)\}$ has at least
    		one accumulation point as $t \to 0$. Since $f(\gamma_x(t)) = t \to 0$, every
    		such accumulation point belongs to $F_0 = f^{-1}(0)$. We claim the limit is
    		unique.
    		
    		\noindent
    		Suppose $a,b \in F_0$ are two distinct accumulation points of
    		$\gamma_x(t)$. Because $X$ is a topologically stratified space and $F_0$ is a closed topologically
    		stratified subspace, every point of $F_0$ admits a distinguished neighborhood. Hence we can choose disjoint open
    		neighborhoods $U_a, U_b \subset X$ of $a,b$ such that $U_a \cap f^{-1}((0,1]) \cong (U_a \cap F_1) \times (0,\varepsilon)$, $U_b \cap f^{-1}((0,1]) \cong (U_b \cap F_1) \times (0,\varepsilon)$ for a small $\varepsilon > 0$. Set $V_a:=U_a \cap F_1$ and $V_b := U_b \cap F_1$.
    		Since $a$ is an accumulation point, $\gamma_x(t)$ lies in
    		$V_a \times (0,\varepsilon)$ for arbitrarily small $t$, which forces
    		$x \in V_a$. The same argument with $b$ gives $x \in V_b$, so
    		$V_a \cap V_b \neq \varnothing$. But then for all sufficiently small $t$,
    		$\gamma_x(t)$ belongs to both $U_a$ and $U_b$, contradicting their
    		disjointness. Hence the limit
    		\begin{align*}
    		g(x) := \lim_{t \to 0^{+}} \varphi(x,t)
    	    \end{align*}
    		is well-defined and takes values in $F_0$.

    		\noindent
    		We still need to show that the map $g:F_{1} \longrightarrow F_{0}$ is continuous. Now consider the graph of $\varphi$ inside the compact Hausdorff space $F_{1}\times X$:
    		\begin{align*}
    		\Gamma_{\varphi}= \{ (x,\varphi(x,t)) \; \vert \; x\in F_{1},\; t\in(0,1] \}.	
    		\end{align*}
    		Let $\overline{\Gamma}_{\varphi}$ be its closure. Take any $(x,y)\in\overline{\Gamma}_{\varphi}\cap (F_{1}\times F_{0})$.  
    		There exists a net $(x_{\alpha},t_{\alpha})\in F_{1}\times(0,1]$ such that $(x_{\alpha},\varphi(x_{\alpha},t_{\alpha}))\to (x,y)$.  
    		Applying the continuous map $f$ to the second coordinate gives $t_{\alpha}=f(\varphi(x_{\alpha},t_{\alpha}))\to f(y)=0$, hence $t_{\alpha}\to0^{+}$.  
    		By the definition of $g$ and the uniqueness of the limit, the image point $y$ must be exactly $g(x)$.  
    		Thus $\overline{\Gamma}_{\varphi}\cap (F_{1}\times F_{0}) = \{\, (x,g(x)) : x\in F_{1} \,\}= \Gamma_{g}$, the graph of $g$.

    		\noindent
    		Since $\overline{\Gamma}_{\varphi}$ is closed in $F_{1}\times X$ and $F_{1}\times F_{0}$ is closed, the graph $\Gamma_{g}$ is closed in $F_{1}\times F_{0}$.  
    		As $F_{1}$ and $F_{0}$ are compact Hausdorff, the closed graph theorem implies that $g$ is continuous

    		\noindent
    		Form the mapping cylinder of $g$, $\cyl(g) = \bigl(F_1 \times [0,1] \;\sqcup\; F_0\bigr)\big/\!\sim$, $(x,0) \sim g(x)$. Define a map $\Phi\colon \cyl(g) \to X$ by $\Phi(x,t) = \varphi(x,t) \;\; (t>0)$, $\Phi(y) = y \;\; (y \in F_0)$. By the definition of $g$, $\Phi$ is continuous at points with $t=0$ and is clearly continuous on elsewhere. Because $F_0$ and $F_1$ are compact,
    		$\cyl(g)$ is compact; $X$ is Hausdorff (as a stratified space). A continuous bijection from a compact space to a Hausdorff space is a
    		homeomorphism. Thus $\Phi$ is a homeomorphism.
            \end{proof}

            The above lemma suggests that, if we locally control the topology of the preimages of the closures of the strata of the refined stratification on $Y$ obtained in Proposition \ref{RefinedStra}, then it may be possible to show that the pullback of the refined stratification on $Y$ under the almost covering $f$ yields a refinement of the trivial stratification on the topological manifold $X$. To be more precise, by using the same notation as in Proposition \ref{RefinedStra}, we require that each point in $W_{\eta} \cap V^{\beta}_{\gamma} \cap S_{\alpha}$ has a small open neighborhood $U$ in $Y$ such that the space $\overline{f^{-1}(U)}$ is a stratified space. 
    
           \begin{definition}\label{SuiAlCov}
        	In the situation of Proposition \ref{RefinedStra}, we call the almost covering $f:X \longrightarrow Y$ \textbf{suitable} if, for all $\eta$, $\beta$, and $\gamma$, each point in $(W_{\eta} \cap V^{\beta}_{\gamma} \cap S_{\alpha})$ has a small open neighborhood $U$ such that the space $\overline{f^{-1}(U)}$ is topologically stratified.
            \end{definition}
    
        \begin{remark}
        	Unless otherwise specified, the term almost covering will refer to a suitable almost covering throughout the remainder of this work.
    \end{remark}

    One can show that if the almost covering $f:X \rightarrow Y$ is suitable, in the sense of above definition, there is a refined stratification on $X$, with respect to which the map $f$ is stratified. However, the existence of such a stratification is not needed for this work. In fact, Lemma~\ref{RefinedStra} shows more than what is needed for this work. Let the filtration $Y \supset Y_{n_{R}}(=R) \supset Y_{n_{R}-1} \supset \cdots$ be the refined stratification on $Y$ obtained by Proposition~\ref{RefinedStra}, where $n_{R}:=\dim(R)$, and set $S_{k}:=Y_{k} \setminus Y_{k-1}$. For our purposes, it suffices to use Lemma~\ref{FibersMap} inductively and argue that if $r \in S_{k}$ for each fiber $f^{-1}(r) \subset f^{-1}(S_{k} )$ there exists an open neighborhood in $U$ which deformation retracts to $f^{-1}(r)$.

      \section{A Decomposition theorem for almost covering maps with non-singular domains}\label{DecomSec}

      With the required topological preliminaries in place, we are now ready to prove the main results of this work. This section is structured as follows. We start by proving a decomposition theorem for a simple class of almost coverings. To be more precise, we study almost coverings such that the restriction to the fibered set is a fiber bundle. We then show that an analogous decomposition theorem holds for a more general class of almost covering maps. However, before we start this section we need the following elementary results.
      
      \begin{proposition}[Associated local system to a covering map]\label{PropDecomp1}
      	Let $f:X \longrightarrow Y$ be a covering map of Hausdorff, path-connected, locally contractible 
      	and locally compact topological spaces of degree $n$. Then in $\operatorname{Sh}(Y)$, we have 
      	\begin{align}\label{Decomp1}
      		f_{\ast} \underline{\mathbb{Q}}_{X} \cong \underline{\mathbb{Q}}_Y \oplus \mathcal{L},
      	\end{align}
      	where $\mathcal{L}$ is a local system of rank $n-1$.
      \end{proposition}
      \begin{proof}
      	Consider the unit morphism under the $(f^{\ast},f_{\ast})$ adjunction $\underline{\mathbb{Q}}_{Y}\longrightarrow f_{\ast} f^{\ast} \underline{\mathbb{Q}}_{Y}$. For a surjective finite map, the unit morphism is injective. Recall that $f^{\ast} \underline{\mathbb{Q}}_{Y} \cong \underline{\mathbb{Q}}_{X}$ in $\operatorname{Sh}(X)$. Hence, we get an injective morphism of sheaves, $\eta : \underline{\mathbb{Q}}_{Y} \longrightarrow f_{\ast}\underline{\mathbb{Q}}_{X} $. Note that since the covering map is unramified, the sheaf $f_{\ast} \underline{\mathbb{Q}}_{X}$ is a local system of rank $n$. Let $\{U_{i}\}_{i \in I}$ be a good cover of $Y$. Then, the morphism $\eta_{U}: \Gamma(U;\underline{\mathbb{Q}}_{Y}) \longrightarrow \Gamma(U; f_{\ast} \underline{\mathbb{Q}}_{X})$ is given by $\eta_{U}(q)=(q, \dots ,q)$ for $U \in \{U_{i}\}_{i \in I}$. We also consider the counit morphism $\varepsilon:f_{\ast}f^{\ast} \underline{\mathbb{Q}}_{Y}(\cong f_{\ast} \underline{\mathbb{Q}}_{X}) \longrightarrow \underline{\mathbb{Q}}_{Y}$. The section morphism $\varepsilon_{U}:\Gamma(U;f_{\ast}\underline{\mathbb{Q}}_{X}) \longrightarrow \Gamma(U; \underline{\mathbb{Q}}_{Y})$ is given by $\varepsilon_{U}(q_{1}, \dots, q_{n})= \sum_{i=1}^{n}q_{i}$. It follows immediately that $\varepsilon_{U} \circ \eta_{U} = n \cdot \operatorname{id}_{\Gamma(U;\underline{\mathbb{Q}}_{Y})}$. As a result, the short exact sequence of sheaves
      	\begin{align*}
      		0 \rightarrow \operatorname{ker}(\varepsilon) \rightarrow f_{\ast} \underline{\mathbb{Q}}_{X} \xrightarrow{\varepsilon} \underline{\mathbb{Q}}_{Y} \rightarrow 0,
      	\end{align*}  
      	splits and the claim follows.
      \end{proof}
        In the following, we refer the local systems $\underline{\mathbb{Q}}_Y \oplus \mathcal{L}$ appearing in Isomorphism \ref{Decomp1} as the associated local system to the covering $f$.
        
        The existence of the canonical truncation morphism, which we establish in the following elementary lemma, is a crucial part of the upcoming discussion. 
        
        \begin{lemma}[Canonical truncation morphism]\label{CanonMorph}
        	Let $f : X \longrightarrow Y$ be a continuous map of topological spaces and let
        	$Rf_* : D(X) \to D(Y)$ be the right derived direct image
        	functor between the unbounded derived categories of sheaves of abelian groups.
        	For every integer $k$ there exists a unique natural transformation
        	\begin{align*}
        \gamma_\mathcal{F}:	\tau_{\leq k} \circ Rf_*  \longrightarrow  Rf_* \circ \tau_{\leq k}	
        	\end{align*}
        	of functors $D(X) \longrightarrow D(Y)$, which makes the following diagram commutative
        	 	\begin{align*}
        		\begin{array}{ccccc}
        			\tau_{\leq k}Rf_*\mathcal{F} & \longrightarrow & Rf_*\mathcal{F} & \longrightarrow &
        			\tau_{\geq k+1}Rf_*\mathcal{F} \\
        			\downarrow{\gamma_\mathcal{F}} & & \parallel & & \\
        			Rf_*\tau_{\leq k}\mathcal{F} & \longrightarrow & Rf_*\mathcal{F} & \longrightarrow &
        			Rf_*\tau_{\geq k+1}\mathcal{F}
        		\end{array}	
        	\end{align*}
        \end{lemma}
        
        \begin{proof}
        	The standard $t$-structure on $D(X)$ (resp.\ $D(Y)$) has truncation
        	functors $\tau_{\leq k}$ and $\tau_{\geq k}$. The functor $f_*$ on abelian sheaves is
        	left exact, so its right derived functor $Rf_*$ satisfies
        	$Rf_*\bigl(\mathrm{D}^{\geq k+1}(X)\bigr) \subseteq \mathrm{D}^{\geq k+1}(Y)$.

        	\noindent
        	For any $\mathcal{F} \in \mathrm{D}(X)$ we have the canonical distinguished triangle
        	$\tau_{\leq k}\mathcal{F} \longrightarrow \mathcal{F} \longrightarrow \tau_{\geq k+1}\mathcal{F} \longrightarrow$
        	in $D(X)$. Applying $Rf_*$ yields a distinguished triangle in $D(Y)$
        	\begin{align*}
             Rf_*\tau_{\leq k}\mathcal{F} \longrightarrow Rf_*\mathcal{F} \longrightarrow
        		Rf_*\tau_{\geq k+1}\mathcal{F} \longrightarrow        		
        		\end{align*}
        	with $Rf_*\tau_{\geq k+1}\mathcal{F} \in D^{\geq k+1}(Y)$.

        	\noindent
        	The object $Rf_*\mathcal{F}$ also admits the truncation triangle
        	\begin{align*}
        		\tau_{\leq k}Rf_*\mathcal{F} \longrightarrow Rf_*\mathcal{F} \longrightarrow
        		\tau_{\geq k+1} Rf_*\mathcal{F} \longrightarrow .	
            \end{align*}
        	Composing the canonical morphisms $\tau_{\leq k}Rf_*\mathcal{F} \longrightarrow Rf_*\mathcal{F}$
        	and $Rf_*\mathcal{F} \longrightarrow Rf_*\tau_{\geq k+1}\mathcal{F}$ results in the canonical morphism
        	$\tau_{\leq k}Rf_*\mathcal{F} \longrightarrow Rf_*\mathcal{F} \longrightarrow
        	Rf_*\tau_{\geq k+1}\mathcal{F}$, which is a morphism from an object of $D^{\leq k}(Y)$ to an object of
        	$D^{\geq k+1}(Y)$.  Since $\operatorname{Hom}(\mathrm{D}^{\leq k}, D^{\geq k+1}) = 0$,
        	this composition vanishes.
        	
        	Hence the map $\tau_{\leq k}Rf_*\mathcal{F} \longrightarrow Rf_*\mathcal{F}$ factors through the weak kernel of
        	$Rf_*\mathcal{F} \longrightarrow Rf_*\tau_{\geq k+1}\mathcal{F}$, which is exactly
        	$Rf_*\tau_{\leq k}\mathcal{F}$.  Consequently we obtain a lift
        	\begin{align*}
        	\gamma_\mathcal{F} : \tau_{\leq k} Rf_*\mathcal{F} \longrightarrow Rf_*\tau_{\leq k}\mathcal{F}    		
        	\end{align*}
        	such that the diagram
        	\begin{align*}
        	\begin{array}{ccccc}
        		\tau_{\leq k}Rf_*\mathcal{F} & \longrightarrow & Rf_*\mathcal{F} & \longrightarrow &
        		\tau_{\geq k+1}Rf_*\mathcal{F} \\
        		\downarrow{\gamma_\mathcal{F}} & & \parallel & & \\
        		Rf_*\tau_{\leq k}\mathcal{F} & \longrightarrow & Rf_*\mathcal{F} & \longrightarrow &
        		Rf_*\tau_{\geq k+1}\mathcal{F}
        	\end{array}	
        	\end{align*}
        	commutes.  Any two such lifts differ by an element of
        	$\operatorname{Hom}\bigl(\tau_{\leq k} Rf_*\mathcal{F}, (Rf_*\tau_{\geq k+1}\mathcal{F})[-1]\bigr)$.
        	The shifted object $(Rf_*\tau_{\geq k+1}\mathcal{F})[-1]$ lies in $\mathrm{D}^{\geq k+2}(Y)$ while the
        	source stays in $\mathrm{D}^{\leq k}(Y)$, so this Hom‑group also vanishes. Thus $\gamma_{\mathcal{F}}$
        	is unique.
        	
        	All ingredients (truncation triangles, $Rf_*$, the lifting argument) are natural in
        	$\mathcal{F}$, so the morphisms $\gamma_\mathcal{F}$ assemble into a natural transformation
        	$\tau_{\leq k}Rf_* \Rightarrow Rf_*\tau_{\leq k}$, which is canonical
        	by its uniqueness.
        \end{proof}

      \subsection{Almost Covering  Maps with Fibered Set Containing one Fiber Bundle} \label{DecompSub1}

       In the following, let $f:X \longrightarrow Y$ be a suitable almost covering map in the sense of Definition \ref{SuiAlCov}. We assume that, the restriction $f \vert_{f^{-1}(R)}$ is a fiber bundle with fiber $F$, where $R$ is the fibered set of the map $f$.

      To begin with, we assume that the spaces $R$ and $F$, and hence consequently $B:=f^{-1}(R)$, are closed topological manifolds. As a result, it follows that the filtrations $X \supset B$ and $Y \supset R$ are topological stratifications. Let $c_{B}:=\operatorname{codim}_{X}(B) \geq 2$ and $c_{R}:=\operatorname{codim}_{Y}(R)$ and consider the following commutative diagram
      \begin{align}\label{DiaOneFib}
		\begin{tikzcd}[ampersand replacement=\&]
			B \arrow[r,  hookrightarrow,"k"] \arrow[d," f \vert_{B}"] \& X \arrow[d,"f"] \arrow[r, hookleftarrow,"i"] \& X \setminus B   \arrow[d,"f \vert_{X \setminus B}"] \\
			R \arrow[r, hookrightarrow,"h"]  \& Y \arrow[r, hookleftarrow,"j"] \& Y \setminus R
		\end{tikzcd}.
       \end{align}
       Since $f \vert_{X \setminus B}$ is a covering map, from Proposition \ref{PropDecomp1} follows that $R(f \vert_{X \setminus B})\underline{\mathbb{Q}}_{X \setminus B} \simeq \underline{\mathbb{Q}}_{Y \setminus R} \oplus \mathcal{L}$, where $\mathcal{L}$ is a local system on $Y \setminus R$.
       In the following, we aim to provide a sufficient topological condition under which the shifted derived direct image complex $Rf_{\ast}\underline{\mathbb{Q}}_{X}[n]$ is quasi-isomorphic to the direct sum $IC_{Y}(\underline{\mathbb{Q}}_{Y \setminus R} \oplus \mathcal{L}) \bigoplus \tau_{\leq \overline{n}(c_{R})-n}Rh_{\ast}(R(f\vert_{B})_{\ast} \underline{\mathbb{Q}}_{B}[n-c_{B}])$. Consider the following distinguished triangle in the derived category $D_{b}^{c}(X)$
       \begin{align*}
	   Rk_{\ast}k^{!} \underline{\mathbb{Q}}_{X} \longrightarrow \underline{\mathbb{Q}}_{X} \longrightarrow Ri_{\ast}i^{\ast}\underline{\mathbb{Q}}_{X} \xrightarrow{[1]}.
       \end{align*}
       As mentioned, $B$ does not posses a tubular neighborhood in $X$, in general. We assume the existence of the associated link bundle to $B$ in $X$ with fibers homeomorphic to a sphere. Note that in a purely topological setting, the homeomorphic type of the link is not well defined. See for example \cite{cannon1979} by Cannon, which shows that the double suspension of every homology sphere is a topological sphere. Furthermore, we assume that the associated link bundle, equivalently the normal bundle of $B$ in $X$, is orientable. This implies that $k^{!}\underline{\mathbb{Q}}_{X} \simeq \underline{\mathbb{Q}}_{B}[-c_{B}]$. Using $i^{\ast}\underline{\mathbb{Q}}_{X} \simeq \underline{\mathbb{Q}}_{X \setminus B}$, we simplify the previous distinguished triangle to 
       
       \begin{align}\label{DisTri1}
	   Rk_{\ast} \underline{\mathbb{Q}}_{B}[-c_{B}] \longrightarrow \underline{\mathbb{Q}}_{X} \longrightarrow Ri_{\ast}\underline{\mathbb{Q}}_{X \setminus B} \xrightarrow{[1]}.
       \end{align}
       Note that if $B \neq \emptyset $ the above distinguished triangle does not split, since a splitting would imply that $\underline{\mathbb{Q}}_{X}$ is quasi-isomorphic to the direct sum $Ri_{\ast}\underline{\mathbb{Q}}_{X \setminus B} \oplus Rk_{\ast}\underline{\mathbb{Q}}_{B}[-c_{B}]$, which leads to a contradiction by taking the sheaf cohomology of the complexes. 
       
       The following elementary lemma provides an alternative way to show that the previous distinguished triangle does not split.

       \begin{lemma}[Top-degree connecting homomorphism]\label{TopConnHom}
       	Let $\pi : E \longrightarrow B$ be an oriented spherical bundle with fibre $S^k$, $k \ge 0$.
       	Suppose that $B$ has finite cohomological dimension, i.e., there exists an integer $N$ such that $H^{i}(B) = 0$ for all $i > N$, and $H^{N}(B) \neq 0$. Then the connecting homomorphism of the associated Gysin sequence $\delta \colon H^{N+k}(E) \longrightarrow H^{N}(B)$
       	is surjective. In particular, it is non‑zero.
       \end{lemma}
       
       \begin{proof}
       	Consider the Gysin sequence of the spherical bundle $S^k \to E \xrightarrow{\pi} B$ with coefficients in $R$:
       	\[
       	\cdots \longrightarrow H^{i}(B) \xrightarrow{e\cup} H^{i+k+1}(B) \xrightarrow{\pi^*} H^{i+k+1}(E) \xrightarrow{\delta} H^{i+1}(B) \longrightarrow \cdots,
       	\]
       	where $e \in H^{k+1}(B)$ is the Euler class. Set $i = N-1$. The sequence therefore reduces to $ H^{N+k}(E) \xrightarrow{\delta} H^{N}(B) \xrightarrow{e\cup} 0$. It follows that the connecting homomorphism $\delta$ is surjective, and in particular a non‑zero map, since $H^{N}(B) \neq 0$.
       \end{proof}
       
       Consider Distinguished triangle~\ref{DisTri1}. Let $x \in B$ be a point,  $U$ a small open neighborhood of $x$ in $X$, and $n_{B}:=\dim_{X}(B)$. Consider the trivial disc bundle $\pi: \overline{U} \longrightarrow \overline{U \cap B}$, which is of rank $c_{B}:=n-n_{B}$. The above lemma shows that the connecting homomorphism $ H^{n-1}(\overline{U} \setminus \overline{U \cap B}) \xrightarrow{ \delta} H^{n_{B}}(\overline{U \cap B})$ is non-trivial and hence the morphism $Ri_{\ast}\underline{\mathbb{Q}}_{X \setminus B} \xrightarrow{[1]} Rk_{\ast}\underline{\mathbb{Q}}_{B}[c_{B}-1]$ is non-trivial. Hence, the distinguished triangle does not split.
       
       As a corollary of Lemma \ref{TopConnHom}, it follows that even after applying the functor $Rf_{\ast}$ to Distinguished triangle \ref{DisTri1}, the resulting distinguished triangle does not split.

       \begin{lemma}\label{SplittingLemma1}
	   Let $f:X \longrightarrow Y$ be an almost covering map from a closed topological $n$-manifold $X$ to a closed topological $n$-pseudomanifold $Y$. Let $R$ be the fibered set of $f$ and suppose the restriction $f|_{f^{-1}(R)}:f^{-1}(R)\longrightarrow R$ is a fiber bundle with fiber $F$. Set $B:=f^{-1}(R)$. Assume $R$ and $F$ are closed topological manifolds, the normal bundle of $B$ in $X$ exists and is orientable, and the link of every $x\in B$ in $X$ is homeomorphic to a sphere. Let $k:B\hookrightarrow X$ and $i:X\setminus B\hookrightarrow X$ be the inclusion maps. Then the distinguished triangle
	   \begin{align}\label{DisTri2}
	   	Rf_{\ast}Rk_{\ast}\underline{\mathbb{Q}}_{B}[-c_B]
	   	\longrightarrow
	   	Rf_{\ast}\underline{\mathbb{Q}}_{X}
	   	\longrightarrow
	   	Rf_{\ast}Ri_{\ast}\underline{\mathbb{Q}}_{X\setminus B}
	   	\xrightarrow{[1]}
	   \end{align}
	   does not split.
        \end{lemma}

        \begin{proof}
	    Let $r \in R$ be a point. Choose a small neighborhood $V$ of $r$ in $Y$, such that $f^{-1}(V \cap R) \cong (V \cap R) \times F$. Note that we have the following commutative diagram
	    \begin{align*}
		\begin{tikzcd}[ampersand replacement=\&]
			f^{-1}(V) \arrow[r] \arrow[d," f \vert_{f^{-1}(V)}"] \& f^{-1}(V \cap R) \arrow[d,"f^{-1} \vert_{f^{-1}(V \cap R)}"]  \\
			V \arrow[r]  \& V \cap R 
		\end{tikzcd},
	    \end{align*}
        where the horizontal arrows are link bundles induced from the stratifications $X \supset B$, and $Y \supset R$. Since $X$ is a topological manifold and the restriction $f \vert_{X \setminus B}$ is a covering map and hence a local homeomorphism, the map $f^{-1}(V) \longrightarrow f^{-1}(V \cap R)$ is a disc bundle of rank $c_{B}:= \operatorname{codim}_{X}(B)$. By the discussion preceding the lemma, Thom isomorphism yields
        \begin{align*}
        H^{k}(\overline{f^{-1}(V \cap R)}) \xrightarrow{\cong} H^{k+c_{B}}(\overline{f^{-1}(V)},\overline{f^{-1}(V)\setminus f^{-1}(V \cap R)}).
        \end{align*}
        Using Thom isomorphism, the long exact sequence of the cohomology groups of the pair on the right-hand side of Thom isomorphism (Gysin sequence) reads
       \begin{align*}
	   \cdots \longrightarrow H^{l}(\overline{f^{-1}(V)}) \longrightarrow H^{l}(\overline{f^{-1}(V) \setminus f^{-1}(V \cap R)}) \longrightarrow H^{l-(c_{B}+1)}(\overline{f^{-1}(V \cap R)}) \longrightarrow \cdots .
       \end{align*}
       Taking the stalks of the distinguished triangle at $r$, namely
       \begin{align*}
	   (Rf_{\ast}Rk_{\ast} \underline{\mathbb{Q}}_{B}[-c_{B}])^{l}_{r} \longrightarrow (Rf_{\ast}\underline{\mathbb{Q}}_{X})^{l}_{r} \longrightarrow (Rf_{\ast}Ri_{\ast}\underline{\mathbb{Q}}_{X \setminus B})^{l}_{r} \xrightarrow{[1]}.
       \end{align*}
       yields the above Gysin sequence. Note that removing the zero section of the normal bundle of $B$ in $X$ and restricting the resulting spherical bundle to $f^{-1}(r)$ yields a $(c_{B}-1)$-spherical bundle $S(f^{-1}(r)) \longrightarrow f^{-1}(r)$. Furthermore, we have $S(f^{-1}(r)) \simeq \overline{f^{-1}(V) \setminus f^{-1}(V \cap R)}$. Lemma \ref{TopConnHom} implies that the connecting homomorphism $H^{l}(S(f^{-1}(r))) \longrightarrow H^{l-(c_{B}+1)}(f^{-1}(r))$ is non-zero, at least when $l$ is the top non-zero degree. Hence the morphism  $Rf_{\ast}Ri_{\ast}\underline{\mathbb{Q}}_{X \setminus B} \longrightarrow Rf_{\ast}Rk_{\ast} \underline{\mathbb{Q}}_{B}[-c_{B}+1]$ is non-zero.
       \end{proof}

        Using the same idea as in the proof of the previous lemma, one may reduce the study of the morphism $Rf_{\ast}Ri_{\ast}\underline{\mathbb{Q}}_{X \setminus B} \longrightarrow Rf_{\ast}Rk_{\ast} \underline{\mathbb{Q}}_{B}[-c_{B}+1]$ to analyzing Gysin sequence of the spherical bundle on $f^{-1}(U) \subset B$, for $U \subset R$. This bundle is induced by restricting the spherical bundle over $B$, which in turn comes from the normal bundle of $B$ in $X$; here $U$ is a trivializing neighborhood of $r$. Furthermore, although Lemma \ref{TopConnHom} shows that, for a given spherical bundle, the connecting homomorphism of the associated Gysin sequence in the top degree is necessarily nonzero, the connecting homomorphisms in other degrees may still vanish. Hence, it is possible that the distinguished triangle \ref{DisTri2}, after applying a well-chosen truncation, splits. In what follows, we develop this idea.

        The following lemma is the cornerstone of the proofs of the main results of this work. Recall that for a given morphism $\alpha: A \rightarrow B$ in the derived category $\mathcal{D}$ of an abelian category $\mathcal{A}$, even if all induced morphisms $\alpha^{i}:\mathcal{H}^{i}(A)\rightarrow \mathcal{H}^{i}(B)$ are zero, the morphism $\alpha$ need not be zero. On the other hand, the topology of a given situation can often only determine whether the induced morphisms $\alpha^{i}$ are zero or not. However, in specific cases, the algebraic structure that comes to our aid here is the orthogonality of the $t$-structure on the derived category $\mathcal{D}$. To be more precise, if $X \in \operatorname{obj}(\mathcal{D}^{\leq m})$ and $Y \in \operatorname{obj}(\mathcal{D}^{\geq m})$, the orthogonality implies that there is a natural isomorphism $\operatorname{Hom}_{\mathcal{D}}(X,Y)\xrightarrow{\cong} \operatorname{Hom}_{\mathcal{A}}(\mathcal{H}^{m}(X),\mathcal{H}^{m}(Y))$. Hence, in this case, $\alpha^{m}$ being zero implies that the morphism $\alpha$ is zero. The following lemma uses this observation and answers the question of whether the distinguished triangle \ref{DisTri2} splits after applying a suitable truncation.

       \begin{lemma}\label{DecompAlg}
       Let $A \xrightarrow{\alpha} D \xrightarrow{\beta} C \xrightarrow{h} A[1]$ be a distinguished triangle in the derived category $\mathcal{D}$ of an abelian category $\mathcal{A}$. Let $n \in \mathbb{Z}$ and suppose:
       	\begin{enumerate}
       		\item $\mathcal{H}^i(A) = 0$ for all $i \le n-1$ (hence $A \in \mathcal{D}^{\ge n}$);
       		\item The induced morphism on cohomology $h^{n-1} : \mathcal{H}^{n-1}(C) \longrightarrow \mathcal{H}^{n}(A)$ is zero, and the morphism $h^n:\mathcal{H}^{n}(C) \longrightarrow \mathcal{H}^{n+1}(A)$ is an isomorphism;
       		\item $\mathcal{H}^i(D) = 0$ for all $i > n$ (hence $D \in \mathcal{D}^{\le n}$).
       	\end{enumerate}
       	Then there is an isomorphism in the derived category
       	\begin{align*}
         D \simeq \tau_{\le n}A \;\oplus\; \tau_{\le n-1}C .
       \end{align*}
       	In particular $\tau_{\le n}A \simeq \mathcal{H}^n(A)[-n]$, so $D \simeq \mathcal{H}^n(A)[-n] \oplus \tau_{\le n-1}C.$
       	
       \end{lemma}
       
       \begin{proof}
       	Let $\mathcal{D}$ be the derived category of an abelian category with its standard $t$-structure.
       	For an integer $m$, let $\tau_{\le m}$ and $\tau_{\ge m}$ be the truncation functors.
       	Recall that for $X \in \mathcal{D}^{\le m}$ and $Y \in \mathcal{D}^{\ge m}$ the natural map
       	$\operatorname{Hom}_{\mathcal{D}}(X,Y) \rightarrow \operatorname{Hom}_{\mathcal{A}}(\mathcal{H}^m(X),\mathcal{H}^m(Y))$ is an isomorphism (orthogonality of the $t$-structure).
       	
       	\noindent
       	Consider the distinguished truncation triangle $\tau_{\le n-1}C \xrightarrow{i} C \xrightarrow{p} \tau_{\ge n}C \xrightarrow{[1]}$. Note that $D \in \mathcal{D}^{\le n}$ and $\tau_{\ge n}C \in \mathcal{D}^{\ge n}$. The original distinguished triangle induces the long exact sheaf cohomology sequence $\mathcal{H}^n(A) \longrightarrow \mathcal{H}^n(D) \longrightarrow \mathcal{H}^n(C) \xrightarrow{h^n} \mathcal{H}^{n+1}(A) \longrightarrow 0$.
       	Since $h^n$ is an isomorphism and $h^{n-1}=0$, the morphism $\mathcal{H}^n(D) \longrightarrow \mathcal{H}^n(C)$ is zero. In particular $\mathcal{H}^n(\beta)=0$, so $p \circ \beta:D \rightarrow \tau_{\geq n}C$ induces the zero map on $\mathcal{H}^n$ and is therefore zero, since $\operatorname{Hom}_{\mathcal{D}}(D,\tau_{\geq n}C) \xrightarrow{\cong} \operatorname{Hom}_{\mathcal{A}}(\mathcal{H}^n(D),\mathcal{H}^n(\tau_{\geq n}C))$. Applying the $\operatorname{Hom}$ functor to the distinguished truncation triangle implies that there exists a unique morphism $f \colon D \to \tau_{\le n-1}C$ such that $i \circ f = \beta$.
       	
       	\noindent
       	Note that the morphism $f:D \longrightarrow \tau_{\le n-1}C$ can be completed to a distinguished triangle
       	$D \xrightarrow{f} \tau_{\le n-1}C \longrightarrow Z \longrightarrow D[1]$.  
       	Rotating this triangle twice and shifting by $[-1]$ gives $Z[-1](=:X) \xrightarrow{j} D \xrightarrow{f} \tau_{\le n-1}C \to Z$, which is again distinguished. From the long exact sheaf cohomology sequence of this triangle and the fact that $f$ (like $\beta$) induces isomorphisms on $\mathcal{H}^i$ for $i \le n-1$,
       	we obtain $\mathcal{H}^i(X)=0$ for $i \neq n$ and $\mathcal{H}^n(X) \cong \mathcal{H}^n(D)$ via $j$.
       	Thus $X \in \mathcal{D}^{\le n} \cap \mathcal{D}^{\ge n}$, so $X \cong \mathcal{H}^n(X)[-n]$.

       	\noindent
       	From the distinguished triangle $A \xrightarrow{\alpha} D \xrightarrow{\beta} C \xrightarrow{h} A[1]$ and the assumptions on $h$, we have that $
       	\mathcal{H}^i(D)\xrightarrow{\mathcal{H}^i(\beta)} \mathcal{H}^i(C)$ is an isomorphism for all $i\le n-1$. Because $\beta = i\circ f$ and the morphism $i$ is an isomorphism on $\mathcal{H}^i$ for $i\le n-1$, the morphism $
       	\mathcal{H}^i(f) : \mathcal{H}^i(D)\longrightarrow \mathcal{H}^i(\tau_{\le n-1}C)$
       	is an isomorphism for all $i\le n-1$.

       	\noindent
         By the previous step, the induced map $\tau_{\le n-1}f\colon \tau_{\le n-1}D \longrightarrow \tau_{\le n-1}C$ is an isomorphism. Let $\phi = \tau_{\le n-1}f$ and let $\iota\colon \tau_{\le n-1}D\to D$ be the canonical morphism. Now define $s : \tau_{\le n-1}C \xrightarrow{\phi^{-1}} \tau_{\le n-1}D \xrightarrow{\iota} D$.
       	Then $f\circ s = f\circ \iota \circ \phi^{-1} = \phi\circ\phi^{-1} = \operatorname{id}_{\tau_{\le n-1}C}$. Thus, the morphism $s$ is a right inverse (section) of $f$. It follows that $D \simeq X \;\oplus\; \tau_{\le n-1}C$.
       	
       	\noindent
       	The long exact cohomology sequence of that triangle, together with the fact that $f$ is an isomorphism on $\mathcal{H}^i$ for $i\le n-1$ and zero on $\mathcal{H}^n$, shows $\mathcal{H}^n(X)\cong \mathcal{H}^n(D)\cong \mathcal{H}^n(A)$, where the last isomorphism comes from the original triangle (since $h^n$ is an isomorphism). As mentioned, $X\in\mathcal{D}^{\le n}\cap\mathcal{D}^{\ge n}$, so $X\cong \mathcal{H}^n(A)[-n]$. Because $A\in\mathcal{D}^{\ge n}$, we have $\tau_{\le n}A\simeq \mathcal{H}^n(A)[-n]$. Therefore $D \simeq \tau_{\le n}A \oplus \tau_{\le n-1}C$,
       	as claimed.

       \end{proof}
        We now have all the necessary tools to prove our first decomposition theorem for the aforementioned simple class of almost coverings.
      
      \begin{proposition}\label{DecoTheo1}
      	Let $f : X \longrightarrow Y$ be an almost covering such that $X$ and the fibered set $R$ are closed topological manifolds, and $Y$ is a closed topological $n$-pseudomanifold. Assume that the restriction $f|_{f^{-1}(R)} : f^{-1}(R) \longrightarrow R$ is a fiber bundle with fiber $F$, a closed orientable topological manifold, and that $\operatorname{codim}_{X}(f^{-1}(R)) \geq 2$. Let $\overline{p} \in \{\overline{n}, \overline{m} \}$ be a perversity. Moreover,
      	\begin{enumerate}
      		\item suppose that the normal bundle of $B:=f^{-1}(R)$ exists and is orientable, and that for each $r \in R$ there is a trivializing neighborhood $U$ of $r$ such that the connecting homomorphism in the Gysin sequence of the normal bundle of $B$ restricted to $f^{-1}(U)$ vanishes in degree $\operatorname{codim}_{X}(B)-1$.
      		\end{enumerate}
      		Let $R(f\vert_{X \setminus B})_{\ast}\underline{\mathbb{Q}}_{X \setminus B} \simeq \underline{\mathbb{Q}}_{Y \setminus R} \oplus \mathcal{L}$, where $\mathcal{L}$ is the local system on $Y \setminus R$ induced by the covering map $f\vert_{X \setminus B}$. If $\overline{p}(\operatorname{codim}_{Y}(R))+1 = \operatorname{codim}_{X}(B)$ and $\operatorname{codim}_{Y}(R)$ is even, then
      	\begin{align}
      		Rf_{\ast}\underline{\mathbb{Q}}_{X}[n] \simeq h_{\ast} \big((R^{0}f\vert_{B})_{\ast} \underline{\mathbb{Q}}_{B}[\dim(B)]\big) \oplus IC^{\overline{p}}_{Y} (\underline{\mathbb{Q}}_{Y \setminus R} \oplus \mathcal{L}),
      	\end{align}
      	where $h: R\hookrightarrow Y$ is the inclusion, and $IC^{\overline{p}}_{Y}$ is Deligne's sheaf complex with respect to the stratification $Y \supset R$.
      \end{proposition}
      
      \begin{proof}
       Let $y \in R$ and set $n_{F}:=\dim(F)$, $c_{R}:=\operatorname{codim}_{Y}(R)$, and $c_{B}:=\operatorname{codim}_{X}(B)$. Using Proposition \ref{RefinedStra}, it follows that the filtration $Y \supset R$ are stratification. Consider the following commutative diagram
      	\begin{align*}
      		\begin{tikzcd}[ampersand replacement=\&]
      			B \arrow[r,  hookrightarrow,"k"] \arrow[d," f \vert_{B}"] \& X \arrow[d,"f"] \arrow[r, hookleftarrow,"i"] \& X \setminus B   \arrow[d,"f \vert_{X \setminus B}"] \\
      			R \arrow[r, hookrightarrow,"h"]  \& Y \arrow[r, hookleftarrow,"j"] \& Y \setminus R
      		\end{tikzcd}.
      	\end{align*}
      	The commutativity of diagram results in the equalities $Rh_{\ast} \circ R(f \vert_{B})_{\ast} \simeq Rf_{\ast} \circ Rk_{\ast}$ and $Rj_{\ast} \circ R(f \vert_{X \setminus B})_{\ast} \simeq Rf_{\ast} \circ Ri_{\ast}$.

      	\noindent
         Note that if $\operatorname{codim}_{Y}(R)$ is even $\overline{p}(c_{R}) = \frac{c_{R}}{2}-1$. As a result, the assumption $\overline{p}(c_{R}) = c_{B}-1$ implies $n_{F} = c_{B}$. Furthermore, since $(R^{i}f_{\ast}\underline{\mathbb{Q}}_{X})_{y}=0$ for $i > n_{F}$ and for all $y \in Y$, Lemma \ref{CanonMorph} shows that the natural morphism $\tau_{\leq c_{B}} Rf_{\ast}\underline{\mathbb{Q}}_{X} \longrightarrow Rf_{\ast} \tau_{\leq c_{B}}\underline{\mathbb{Q}}_{X}(\simeq Rf_{\ast}\underline{\mathbb{Q}}_{X})$ is a quasi-isomorphism.

         \noindent
         Our goal is to apply Lemma \ref{DecompAlg} to the distinguished triangle
         \begin{align*}
         		Rf_{\ast}Rk_{\ast}\underline{\mathbb{Q}}_{B}[-c_B]
         	\longrightarrow
         	Rf_{\ast}\underline{\mathbb{Q}}_{X}
         	\longrightarrow
         	Rf_{\ast}Ri_{\ast}\underline{\mathbb{Q}}_{X\setminus B}
         	\xrightarrow{[1]}.
         \end{align*}
         
         Let $U \subset B$ be a trivializing neighborhood of $r \in R$, i.e., $f^{-1}(U) \cong F \times \mathbb{R}^{\dim(B)}$. Furthermore, let $\pi: E \longrightarrow f^{-1}(U)$ be the associated $(c_{B}-1)$-spherical bundle over $f^{-1}(U)$ in $X$, obtained by restricting the spherical bundle over $B$ to $f^{-1}(U)$. Taking stalks of the above distinguished triangle and using $f^{-1}(U) \simeq F$ results in the following Gysin sequence
         \begin{align*}
         	0 \rightarrow H^{c_{B}-1}(F) \rightarrow H^{c_{B}-1}(E) &\xrightarrow{\delta_{c_{B}-1}(=0)}H^{0}(F)   \rightarrow H^{c_{B}}(F)  \rightarrow H^{c_{B}}(E) \xrightarrow{\delta_{c_{B}}(\cong)}H^{1}(F) \rightarrow 0.
         \end{align*}  
         Note that since $F$ is orientable, it follows that $H^{c_{B}}(F)$ is not trivial and hence the homomorphism $H^{0}(F)   \rightarrow H^{c_{B}}(F)$ is an isomorphism. The exactness yields that $\delta_{c_{B}}$ is an isomorphism.
         It follows that the morphism $\mathcal{H}^{i}(
         Rf_{\ast}Ri_{\ast}\underline{\mathbb{Q}}_{X\setminus B}) \longrightarrow \mathcal{H}^{i+1}(
         Rf_{\ast}Rk_{\ast}\underline{\mathbb{Q}}_{B}[-c_B])$ is zero for $i < c_{B}$ and an isomorphism for $i=c_{B}$.

         \noindent
         Lemma \ref{DecompAlg} implies $Rf_{\ast}\underline{\mathbb{Q}}_{X} \simeq	\tau_{\leq c_{B}} Rf_{\ast}Rk_{\ast}\underline{\mathbb{Q}}_{B}[-c_B] \oplus \tau_{\leq c_{B}-1}
         	Rf_{\ast}Ri_{\ast}\underline{\mathbb{Q}}_{X\setminus B}$. Note that the inclusion $h:R \hookrightarrow Y$ is closed inclusion. Hence, using the commutativity of the above diagram, considering $\overline{p}(c_{R})=c_{B}-1$, and applying a shift by $[n]$ proves the claim.
          \end{proof}

      \begin{remark}
      	The decomposition theorem for algebraic semi-small maps is well known and studied in detail (see, e.g., \cite{deCataldoMigliorini2002}). In the situation of Proposition \ref{DecoTheo1}, if the spaces $X$, $B$, and $R$ have real even dimensions, then the condition $\overline{p}(\operatorname{codim}_{Y}(R))+1 = \operatorname{codim}_{X}(B)$ is equivalent to $2 \dim_{\mathbb{C}}(F)+\dim_{\mathbb{C}}(R)=\dim_{\mathbb{C}}(X)$. The latter condition is the same as requiring the map $f:X \longrightarrow Y$ to behave like a semi-small map with respect to the dimensions of $F$ and $R$. In fact, if the almost covering $f:X \longrightarrow Y$ is semi-small and an algebraic morphism, then the statement of Proposition \ref{DecoTheo1} is equivalent to the decomposition theorem for algebraic almost covering maps.
      \end{remark}

        The above proposition gives a condition under which the shifted derived direct image $Rf_{\ast}\underline{\mathbb{Q}}_{X}[n]$ is quasi-isomorphic to the direct sum of Deligne sheaf complex $IC_{Y}^{\overline{p}}(\underline{\mathbb{Q}}_{Y \setminus R} \oplus \mathcal{L})$ which is determined solely by the covering part of the almost covering and a contribution from the fiber bundle component of the map. However, in the situation of the above proposition, the assumptions on $R$ and $F$ force the filtration $Y \supset R$ to be a stratification. The upcoming theorem answers the decomposition question for arbitrary stratifications $Y \supset Y_{n-2} \supset \cdots$ and $R \supset R_{n-3} \supset \cdots$.

     \begin{theorem}\label{DecoTheo1b}
     	Let $f : X \longrightarrow Y$ be an almost covering such that $X$ is a closed topological manifold, the fibered set $R$ is closed  stratified space, and $Y$ is a closed topological $n$-pseudomanifold. Assume that the restriction $f|_{f^{-1}(R)} : f^{-1}(R) \longrightarrow R$ is a fiber bundle with fiber $F$, a closed orientable topological manifold, and that $\operatorname{codim}_{X}(f^{-1}(R)) \geq 2$. Furthermore, let the filtration $Y \supset Y_{n_{R}}(=R) \supset Y_{n_{R}-1} \supset \cdots$ be the refined stratification on $Y$ obtained by Proposition~\ref{RefinedStra}, where $n_{R}:= \dim(R)$. Set $S_{k}:= Y_{k} \setminus Y_{k-1}$ and assume the following:
     	\begin{enumerate}
     		\item suppose that the normal bundle over $f^{-1}(S_{k})$ in $X$ exists and is orientable, and that for each point in $R \setminus Y_{n_{R} -1}$, there is a trivializing neighborhood $U$ such that the connecting homomorphism in the Gysin sequence of the normal bundle over $f^{-1}(U)$ vanishes in degree $\operatorname{codim}_{X}(B)-1$.
     		\end{enumerate}
     		Let $R(f\vert_{X \setminus B})_{\ast}\underline{\mathbb{Q}}_{X \setminus B} \simeq \underline{\mathbb{Q}}_{Y \setminus R} \oplus \mathcal{L}$, where $\mathcal{L}$ is the local system on $Y \setminus R$ induced by the covering map $f\vert_{X \setminus B}$. If $\overline{p}(\operatorname{codim}_{Y}(R))+1 = \operatorname{codim}_{X}(B)$ and the refined strata of $Y$ are even-codimensional, then
     	\begin{align}
     	Rf_{\ast}\underline{\mathbb{Q}}_{X}[n] \simeq h_{\ast}(IC^{\overline{p}}_{R}\big(R^{0}( f \vert_{f^{-1}(S_{n_{R}}) })_{\ast}\underline{\mathbb{Q}}_{f^{-1}(S_{n_{R}})}\big) \oplus IC_{Y}^{\overline{p}}(\underline{\mathbb{Q}}_{Y \setminus R} \oplus \mathcal{L} ),
     	\end{align}
     	where $h: R\hookrightarrow Y$ is the canonical closed inclusion.
     \end{theorem}
   
     \begin{proof}
     Let the filtration $Y \supset Y_{n_{R}}(=R) \supset Y_{n_{R}-1} \supset \cdots$ be the refined stratification obtained from Proposition~\ref{RefinedStra}. Note that the restriction map $f \vert_{S_{k}}$ is a fiber bundle with fiber $F$, a closed topologcial manifold, for $0<k<n_{R}$. Let $X \supset X_{n_{B}}(=B) \supset X_{n_{B}-1} \supset \cdots$ be the filtration obtained by setting $B_{k}:=f^{-1}(Y_{k+(n_{B}-n_{R})})$, where $n_{B}:=\dim(B)$. Since the fiber is a closed orientable topological manifolds, it follows that the sets $X_{k} \setminus X_{k-1}$ are topological manifolds. Note that, this filtration is not in general a topological stratification.

     \noindent
     Consider the commutative diagram
     	\begin{align*}
     	\begin{tikzcd}[ampersand replacement=\&]
     		B \setminus X_{n_{B}-1} \arrow[r,  hookrightarrow,"k_{2}"] \arrow[d," f \vert_{B \setminus X_{n_{B}-1}}"] \& X \setminus X_{n_{B}-1} \arrow[d,"f \vert_{X \setminus X_{n_{B}-1}}"] \arrow[r, hookleftarrow,"i_{2}"] \& X \setminus B   \arrow[d,"f \vert_{X \setminus B}"] \\
     		R \setminus Y_{n_{R}-1} \arrow[r, hookrightarrow,"h_{2}"]  \& Y \setminus Y_{n_{R}-1} \arrow[r, hookleftarrow,"j_{2}"] \& Y \setminus R
     	\end{tikzcd}.
     \end{align*}
     By assumption, the normal bundle bundle over $B \setminus X_{n_{B}-1}$ in $X$ exists and is orientable. Hence, we obtain the following distinguished triangle
     
     \begin{align*}
     	R(f\vert_{X \setminus X_{n_{B}-1}})_{\ast}Rk_{2 \ast} \underline{\mathbb{Q}}_{ B \setminus X_{n_{B}-1}}[-c_{B}] \rightarrow R(f\vert_{X \setminus X_{n_{B}-1}})_{\ast} \underline{\mathbb{Q}}_{X \setminus X_{n_{B}-1}} \rightarrow R(f\vert_{X \setminus X_{n_{B}-1}})_{\ast}Ri_{2 \ast}\underline{\mathbb{Q}}_{X \setminus B} \xrightarrow{[1]}.
     \end{align*}
 
     An inductive application of Lemma \ref{FibersMap} shows that each point in $B \setminus X_{n_{B}-1}$ has an small open neighbourhood that deformation retracts to $F$. Hence, as in the proof of Proposition \ref{DecoTheo1}, the vanishing assumption on the connecting homomorphisms in Gysin sequence and applying the truncation functor $\tau_{\leq c_{B}}$ to the above distinguished triangle results in the decomposition
     \begin{align*}
     	R(f\vert_{X \setminus X_{n_{B}-1}})_{\ast} \underline{\mathbb{Q}}_{X \setminus X_{n_{B}-1}} \simeq R^{0}(h_{2} \circ f \vert_{B \setminus X_{n_{B}-1}})_{\ast}\underline{\mathbb{Q}}_{B \setminus X_{n_{B}-1}}[-c_{B}] \oplus \tau_{\leq \overline{p}(c_{R})} Rj_{2 \ast} (\underline{\mathbb{Q}}_{Y \setminus R} \oplus \mathcal{L} ),
     \end{align*}
     where we used $\overline{p}(c_{R})=c_{B}-1$ and that the morphism obtained by Lemma~\ref{CanonMorph} is a quasi-isomorphism, since $\dim(F)=c_{B}(:=n-n_{B})$. Recall that $h_{2}$ is a closed inclusion. Hence the above quasi-isomorphism can be rewritten as
      \begin{align*}
     	R(f\vert_{X \setminus X_{n_{B}-1}})_{\ast} \underline{\mathbb{Q}}_{X \setminus X_{n_{B}-1}} \simeq h_{2\ast}(R^{0}(f \vert_{B \setminus X_{n_{B}-1}})_{\ast}\underline{\mathbb{Q}}_{B \setminus X_{n_{B}-1}}[-c_{B}] \oplus \tau_{\leq \overline{p}(c_{R})} Rj_{2 \ast} (\underline{\mathbb{Q}}_{Y \setminus R} \oplus \mathcal{L} ),
     \end{align*}
     
     \noindent
     We now study the following diagram
     	\begin{align*}
     	\begin{tikzcd}[ampersand replacement=\&]
     		X_{n_{B}-1} \setminus X_{n_{B}-2} \arrow[r,  hookrightarrow,"k_{3}"] \arrow[d," f \vert_{X_{n_{B}-1} \setminus X_{n_{B}-2}}"] \& X \setminus X_{n_{B}-2} \arrow[d,"f \vert_{X \setminus X_{n_{B}-2}}"] \arrow[r, hookleftarrow,"i_{3}"] \& X \setminus X_{n_{B}-1}   \arrow[d,"f \vert_{X \setminus X_{n_{B}-1}}"] \\
     		 Y_{n_{R}-1} \setminus Y_{n_{R}-2} \arrow[r, hookrightarrow,"h_{3}"]  \& Y \setminus Y_{n_{R}-2} \arrow[r, hookleftarrow,"j_{3}"] \& Y \setminus Y_{n_{R}-1}
     	\end{tikzcd}.
     \end{align*}
     Since, by the assumption, the normal bundle over the stratum $X_{n_{B}-1} \setminus X_{n_{B}-2}$ in $X$ exists and is orientable, we have the following distinguished triangle
     \begin{align*}
     		R(f\vert_{X_{n_{B}-1} \setminus X_{n_{B}-2}})_{\ast}Rk_{3 \ast} &\underline{\mathbb{Q}}_{ X_{n_{B}-1} \setminus X_{n_{B}-2}}[-c_{B_3}] \rightarrow 	R(f\vert_{X_{n_{B}-1} \setminus X_{n_{B}-2}})_{\ast} \underline{\mathbb{Q}}_{X \setminus X_{n_{B}-2}} \\ &\rightarrow 	R(f\vert_{X_{n_{B}-1} \setminus X_{n_{B}-2}})_{\ast}Ri_{3 \ast}\underline{\mathbb{Q}}_{X \setminus X_{n_{B}-1}} \xrightarrow{[1]},
     \end{align*}
     where $c_{B_{3}}:=\operatorname{codim}_{X}(	X_{n_{B}-1} \setminus X_{n_{B}-2})$. Set $c_{R_{3}}:=\operatorname{codim}_{Y}( Y_{n_{R}-1} \setminus Y_{n_{R}-2})$, and note that $\overline{p}(c_{R_{3}})<c_{B_{3}}-1$. Hence, the natural morphism obtained from Lemma~\ref{CanonMorph} is a quasi-isomorphism and by applying the truncation functor $\tau_{\overline{p}(c_{R_{3}})}$ to the above distinguished triangle, we obtain
     \begin{align*}
     	R(f\vert_{X_{n_{B}-1} \setminus X_{n_{B}-2}})_{\ast} \underline{\mathbb{Q}}_{X \setminus X_{n_{B}-2}} &\simeq  \tau_{\leq \overline{p}(c_{R_{3}})}Rj_{3 \ast}(h_{2\ast}(R^{0}(f \vert_{B \setminus X_{n_{B}-1}})_{\ast}\underline{\mathbb{Q}}_{B \setminus X_{n_{B}-1}}[-c_{B}])\\ &\oplus \tau_{\leq \overline{p}(c_{R_{3}})}Rj_{3 \ast}\tau_{\leq \overline{p}(c_{R})} Rj_{2 \ast} (\underline{\mathbb{Q}}_{Y \setminus R} \oplus \mathcal{L} ).
     \end{align*} 
     Consider the commutative diagram
      	\begin{align*}
      	\begin{tikzcd}[ampersand replacement=\&]
      	R \setminus Y_{n_{R}-1} \arrow[r,hookrightarrow,"h_{2}"] \arrow[d,hookrightarrow,"j_{3}'"]	\& Y \setminus Y_{n_{R}-1} \arrow[d,hookrightarrow,"j_{3}"] \\
      	R \setminus Y_{n_{R}-2} \arrow[r,hookrightarrow,"h_{2}'"]	\& Y \setminus Y_{n_{R}-2} 
      	\end{tikzcd},
      \end{align*}
     and note that the inclusion $h_{2}'$ is closed. It follows that 
      \begin{align*}
     	R(f\vert_{X_{n_{B}-1} \setminus X_{n_{B}-2}})_{\ast} \underline{\mathbb{Q}}_{X \setminus X_{n_{B}-2}} &\simeq h_{2\ast}' \big( \tau_{\leq \overline{p}(c_{R_{3}})}Rj_{3 \ast}(R^{0}(f \vert_{B \setminus X_{n_{B}-1}})_{\ast}\underline{\mathbb{Q}}_{B \setminus X_{n_{B}-1}}[-c_{B}])\big) \\ &\oplus \tau_{\leq \overline{p}(c_{R_{3}})}Rj_{3 \ast}\tau_{\leq \overline{p}(c_{R})} Rj_{2 \ast} (\underline{\mathbb{Q}}_{Y \setminus R} \oplus \mathcal{L} ).
     \end{align*}
     
     By inductive sue of the previous step and performing a shift by $[n]$, we arrive at
     \begin{align*}
     	Rf_{\ast}\underline{\mathbb{Q}}_{X}[n] \simeq h_{\ast}(IC^{\overline{p}}_{R}\big(R^{0}( f \vert_{B \setminus X_{n_{B}-1}})_{\ast}\underline{\mathbb{Q}}_{B \setminus X_{n_{B}-1}}\big) \oplus IC_{Y}^{\overline{p}}(\underline{\mathbb{Q}}_{Y \setminus R} \oplus \mathcal{L} ),
     \end{align*}
     where $h:R \hookrightarrow Y$, which is a closed inclusion.
     \end{proof}

    \subsection{Almost Covering  Maps with Fibered Set Containing more than One Fiber Bundle}
    In this section, we want to study covering maps such that the fibered set contains more than one fiber bundle. However, there are some technical details that we need to investigate before presenting the main theorem of this section.

    Let $f:X \longrightarrow Y$ be an almost covering with the decomposition of the fibered set $R= \bigsqcup_{\beta}  T_{\beta}$, where $T_{\beta}$ is a topological stratified space for each $\beta$. Assume that $X$ is a closed $n$-manifold and $Y$ is a closed $n$-pseudomanifold. As mentioned before, we consider only suitable decompositions in the sense of Definition \ref{SutStrat} (i.e., the conditions in Lemma \ref{StratR} are satisfied). In particular the above decomposition of $R$ should satisfy the frontier condition. Let the closed orientable topological manifold $F_{\beta}$ be fibers of the restriction $f \vert_{T_{\beta}}$ for each $\beta$. Let the filtration $Y \supset Y_{n_{R}}(=R) \supset Y_{n_{R}-1} \supset \cdots$ be the refined stratification on $Y$ obtained by Proposition~\ref{RefinedStra}, where $n_{R}:= \dim(R)$. Set $S_{k}:=Y_{k} \setminus Y_{k-1}$. Note that the restriction $f \vert_{S_{k}}$ is a fiber bundle.

    \noindent
    As in the previous subsection, let $c_{R_{k}}:=\operatorname{codim}_{Y}(S_{k}) \geq 2$. Consider the canonical truncation morphism
    \begin{align*}
    	\tau_{\leq \overline{p}(c_{R_{k}}) }(Rf\vert_{f^{-1}(S_{k})})_{\ast}\underline{\mathbb{Q}}_{f^{-1}(S_{k})} \longrightarrow Rf_{\ast} \tau_{\leq \overline{p}(c_{R_{k}}) } \underline{\mathbb{Q} }_{f^{-1}(S_{k})},
    \end{align*}
    which exists by Lemma \ref{CanonMorph}. Assume that for each $k$ and $r \in S_{k}$ the condition $2 \dim(f^{-1}(r))+\dim(S_{k}) \leq n$ is satisfied. Note that $\dim(f^{-1}(r))+\dim(S_{k}) = \dim(f^{-1}(S_{k}))$, and since the above condition is satisfied for each $k$, it follows that $\dim(f^{-1}(r))) \leq n-\dim(f^{-1}(S_{k}))$.
    An immediate consequence is that the above canonical morphism is a quasi-isomorphism. Let $R_{\operatorname{rel}}:=\{S_{k} \ \vert \  2 \dim(f^{-1}(r))+\dim(S_{k}) = n \ \forall \ r \in f^{-1}(S_{k})\}$. This observation motivates the following definition.

    \begin{definition}\label{RelPart}
    	Let $f:X \longrightarrow Y$ be an almost covering where $X$ is a closed $n$-manifold, $Y$ is a closed $n$-pseudomanifold. Let the filtration $Y \supset Y_{n_{R}}(=R) \supset Y_{n_{R}-1} \supset \cdots$ be the refined stratification on $Y$ obtained by Proposition~\ref{RefinedStra}, where $R$ is the fibered set and $n_{R}:= \dim(R)$. The almost covering $f:X \longrightarrow Y$ is called a \textbf{topologically semi-small map} if for each point $r \in S_{k}(:=Y_{k} \setminus Y_{k-1})$ and for each $0 \leq k \leq n_{R}$, the condition $2 \dim(f^{-1}(r))+\dim(S_{k}) \leq n$ is satisfied. Furthermore, we the set $R_{\operatorname{rel}}:=\{S_{k} \ \vert \  2 \dim(f^{-1}(r))+\dim(S_{k}) = n \ \forall \ r \in f^{-1}(S_{k})\}$ is referred to as the set of relevant strata.
    \end{definition}
    
    To avoid ambiguity, we note that the set of relevant strata is defined slightly differently in algebraic geometry. There, the top stratum of $Y$ in the refined stratification is considered a relevant stratum; however, in the above definition, the term “relevant strata” refers to those relevant strata that are a subset of $R$.

    Having defined the set of relevant strata, we can now present the main result of this work.

    \begin{theorem}\label{DecoTheo3}
    	Let $f : X \rightarrow Y$ be a topologically semi-small map such that $X$ is a closed topological manifolds, the fibered set $R$ is a closed stratified space, and $Y$ is a closed topological $n$-pseudomanifold. Set $B:=f^{-1}(R)$, let $R=\bigsqcup_{\beta}T_{\beta}$ be a suitable decomposition of the fibered set, and let $R_{\operatorname{rel}}$ be its relevant part. Furthermore, let the filtration $Y \supset Y_{n_{R}}(=R) \supset Y_{n_{R}-1} \supset \cdots$ be the refined stratification on $Y$ obtained by Proposition~\ref{RefinedStra}, where $n_{R}:= \dim(R)$. Set $S_{k}:= Y_{k} \setminus Y_{k-1}$, and assume the followings:
    	\begin{enumerate}
    		\item for each $r \in R$ the fiber $f^{-1}(r)$ is a closed orientable topological manifold;
    		\item the normal bundle of the manifold $f^{-1}(S_{k})$ in $X$ exists, for each $0 \leq k \leq n_{R}$, and is orientable, and for each $S_{k}\in R_{\operatorname{rel}}$ and each $r\in S_{k}$ there is a small open neighborhood $U$ of $r$ such that the connecting homomorphism in the Gysin sequence of the normal bundle of $S_{k}$ restricted to $f^{-1}(U)$ vanishes in degree $\operatorname{codim}_{X}(B)-1$.
    			\end{enumerate}
    		Let $R(f|_{X \setminus B})_{*}\,\underline{\mathbb{Q}}_{X \setminus B}\simeq \underline{\mathbb{Q}}_{Y \setminus R}\oplus \mathcal{L}$, where $\mathcal{L}$ is the local system on $Y \setminus R$ associated to the covering map $f|_{X \setminus B}$. If each $\operatorname{codim}_{Y}(S_{k})$ is even and $\operatorname{codim}_{X}(B)\geq 2$, then
    	\begin{align}\label{DecoTheo3Eq}
    		Rf_{\ast}\underline{\mathbb{Q}}_{X}[n]\simeq \bigoplus_{S_{k} \in R_{\operatorname{rel}}}\iota_{k\ast}(IC_{\overline{S_{k}}} (R^{0}f\vert_{f^{-1}(S_{k})}\underline{\mathbb{Q}}_{S_{k}})) \oplus IC_{Y}(\underline{\mathbb{Q}}_{Y \setminus R} \oplus \mathcal{L}),
    	\end{align}
    	where $\iota_{k}:\overline{S_{k}} \hookrightarrow Y$ be the canonical closed inclusion, and $IC^{\overline{p}}_{Y}$ is Deligne's intersection complex with respect to the above refined stratification and the perversity $\overline{p}\in\{\overline{n},\overline{m}\}$.
    \end{theorem}
    \begin{proof}
    	The proof analogous to the proof of Proposition \ref{DecoTheo1}. Let $X \supset X_{n_{B}}(=B) \supset X_{n_{B}-1} \supset \cdots$ be the filtration obtained by setting $B_{k}:=f^{-1}(Y_{k+(n_{B}-n_{R})})$, where $n_{B}:=\dim(B)$. Similar to the above consideration, since the fibers are closed orientable topological manifolds, it follows that the sets $X_{k} \setminus X_{k-1}$ are topological manifolds. Consider the following commutative diagram
    	\begin{align*}
    		\begin{tikzcd}[ampersand replacement=\&]
    			B \setminus X_{n_{B}-1} \arrow[r,  hookrightarrow,"k_{2}"] \arrow[d," f \vert_{B \setminus X_{n_{B}}}"] \& X \setminus B_{n_{B}-1} \arrow[d,"f \vert_{X \setminus X_{n_{B}-1}}"] \arrow[r, hookleftarrow,"i_{2}"] \& X \setminus B   \arrow[d,"f \vert_{X \setminus B}"] \\
    			R \setminus Y_{n_{R}-1} \arrow[r, hookrightarrow,"h_{2}"]  \& Y \setminus Y_{n_{R}-1}  \arrow[r, hookleftarrow,"j_{2}"] \& Y \setminus R
    		\end{tikzcd}.
    	\end{align*}
    	Note that the space $B \setminus X_{n_{B}-1}$ is closed in $X \setminus X_{n_{B}-1}$. Since, by the assumption, the normal bundle of $	B \setminus X_{n_{B}-1} $ in $X$ exists and is orientable, we have the following non-splitting distinguished triangle
    	\begin{align*}
    		Rk_{2\ast} \underline{\mathbb{Q}}_{B \setminus X_{n_{B}-1} }[-c_{B}] \longrightarrow \underline{\mathbb{Q}}_{X \setminus X_{n_{B}-1}} \longrightarrow Ri_{2\ast}\underline{\mathbb{Q}}_{X \setminus B} \xrightarrow{[1]},
    	\end{align*} 
    	where $c_{B}:=\operatorname{codim}_{X}(B \setminus B_{n_{B}-1})$. Applying the functor $(Rf \vert_{X \setminus B_{n_{B}-1}})_{\ast}$ to the above triangle yields 
    	\begin{align*}
    		(Rf \vert_{X \setminus B_{n_{B}-1}})_{\ast} Rk_{2\ast} \underline{\mathbb{Q}}_{B \setminus X_{n_{B}-1} }[-c_{B}] \longrightarrow (Rf \vert_{X \setminus B_{n_{B}-1}})_{\ast}\underline{\mathbb{Q}}_{X \setminus B_{n_{B}-1}} \longrightarrow (Rf \vert_{X \setminus B_{n_{B}-1}})_{\ast}Ri_{2\ast}\underline{\mathbb{Q}}_{X \setminus B} \xrightarrow{[1]}.
    	\end{align*} 
    	If the stratum $R \setminus Y_{n_{R}-1}$ is relevant, it follows that $\overline{p}(c_{R})=c_{B}-1$, where $c_{R}:=\operatorname{codim}_{Y}(R \setminus Y_{n_{R}-1})$. Note that the canonical morphism $\tau_{\leq \overline{p}(c_{R})+1}(f \vert_{X \setminus B_{n_{B}-1}})_{\ast}\underline{\mathbb{Q}}_{X \setminus B_{n_{B}-1}} \longrightarrow (f \vert_{X \setminus B_{n_{B}-1}})_{\ast} \underline{\mathbb{Q}}_{X \setminus B_{n_{B}-1}}$ obtained from Lemma~\ref{CanonMorph} is a quasi-isomorphism. By Lemma \ref{FibersMap} together with local flatness implies that each fiber $f^{-1}(r) \subset B \setminus B_{n_{B}-1}$, where $r \in R \setminus R_{n_{R}-1}$, has a sufficiently small open neighborhood in $X \setminus X_{n_{B}-1}$ homotopy equivalent to $F_{2}$, the fiber of the fiber bundle $f \vert_{ B \setminus X_{n_{B}-1}}$. In particular, we obtain $(Rf\vert_{\ast} \underline{\mathbb{Q}}_{X \setminus X_{n_{B}-1}})_{y}^{i}=0$ for $i > c_{B}$ and for all $y \in Y \setminus Y_{n_{R}-1}$. For a point $r \in R \setminus R_{n_{R}-1}$, let $U \subset R \setminus R_{n_{R}-1}$ be a trivializing neighborhood. Furthermore, let $\pi: E \longrightarrow f^{-1}(U)$ be the associated $(c_{B}-1)$-spherical bundle over $f^{-1}(U)$ in $X$, obtained by restricting the spherical bundle over $B \setminus X_{n_{B}-1}$ to $f^{-1}(U)$. Hence, taking stalks of the above distinguished triangle and using $f^{-1}(U) \simeq F_{2}$ results in the following Gysin sequence
    	\begin{align*}
    		0 \rightarrow H^{c_{B}-1}(F) \rightarrow H^{c_{B}-1}(E) \xrightarrow{\delta_{c_{B}-1}(=0)}H^{0}(F)   \rightarrow H^{c_{B}}(F)  \rightarrow H^{c_{B}}(E) \xrightarrow{\delta_{c_{B}}(\cong)}H^{1}(F) \rightarrow 0.
    	\end{align*}  
    	From the vanishing assumption on the connection homomorphism of Gysin sequence in degree $c_{B}-1$, it follows that the morphism $\mathcal{H}^{i}(
    	Rf\vert_{\ast} Ri_{\ast}\underline{\mathbb{Q}}_{X \setminus B} ) \longrightarrow \mathcal{H}^{i+1}(
    	Rf\vert_{\ast} Rk'_{\ast} \underline{\mathbb{Q}}_{B_{\operatorname{rel}}}[-c_{B}])$ is zero for $i < c_{B}$ and an isomorphism for $i=c_{B}$.

    	\noindent
    	Hence, by applying the truncation functor $\tau_{\leq \overline{p}(c_{R})+1}$ to the above distinguished triangle, employing Lemma \ref{DecompAlg}, and using the commutativity of the diagram, we arrive at
    	\begin{align*}
    		(Rf \vert_{X \setminus X_{n_{B}-1}})_{\ast}\underline{\mathbb{Q}}_{X \setminus X_{n_{B}-1}} \simeq 	 h_{2\ast}(R^{0}f \vert_{B \setminus B_{n_{B}-1}})_{\ast} \underline{\mathbb{Q}}_{B \setminus B_{n_{B}-1} }[-c_{B}] \oplus \tau_{\leq \overline{p}(c_{R})} Rj_{2\ast}(\underline{\mathbb{Q}}_{Y \setminus R} \oplus \mathcal{L}),
    	\end{align*}  
    	where we used that the inclusion $h_{2}$ is closed.

    	\noindent
    	If the stratum $Y_{k} \setminus Y_{k-1}$ is not relevant, we obtain
    	\begin{align*}
    		(Rf \vert_{X \setminus X_{n_{B}-1}})_{\ast}\underline{\mathbb{Q}}_{X \setminus X_{n_{B}-1}} \simeq \tau_{\leq \overline{p}(c_{R})} Rj_{2\ast}(\underline{\mathbb{Q}}_{Y \setminus R} \oplus \mathcal{L}).
    	\end{align*}

    	In the next step, we start by examining the following diagram
    	\begin{align*}
    		\begin{tikzcd}[ampersand replacement=\&]
    			U'_{2}\arrow[r,  hookrightarrow,"k_{3}"] \arrow[d," f \vert_{U_{2}'}"] \& U_{3} \arrow[d,"f \vert_{U_{3}}"] \arrow[r, hookleftarrow,"i_{3}"] \& X \setminus B_{n_{B}-1}    \arrow[d,"f \vert_{X \setminus B_{n_{B}-1} }"] \\
    			V'_{2} \arrow[r, hookrightarrow,"h_{3}"]  \& V_{3} \arrow[r, hookleftarrow,"j_{3}"] \& Y \setminus R_{n_{R}-1}
    		\end{tikzcd},
    	\end{align*}
    	where $U_{3}:= X \setminus B_{n_{B}-2}$, $V_{3}:=Y \setminus R_{n_{R}-2} $, $U'_{2}:=X_{n_{B}-1} \setminus X_{n_{B}-2} $, and $V'_{2}:=Y_{n_{R}-1} \setminus Y_{n_{R}-2}$. Hence, by the same argument as above, we arrive at the following distinguished triangle
    		\begin{align*}
    		(Rf \vert_{U_{3}})_{\ast} Rk_{3\ast} \underline{\mathbb{Q}}_{U'_{2}}[-c_{B_{3}}] \longrightarrow (Rf \vert_{U_{3}})_{\ast}\underline{\mathbb{Q}}_{U_{3}} \longrightarrow (Rf \vert_{U_{3}})_{\ast}Ri_{3\ast}\underline{\mathbb{Q}}_{X \setminus B_{n_{B}-1}} \xrightarrow{[1]},
    	\end{align*} 
    	where $c_{B_3}:=\operatorname{codim}_{X}(U_{2}')$. An inductive application of Lemma~\ref{FibersMap} implies that each fiber $f^{-1}(r) \subset U_{2}'$, where $r \in V'_{2}$, has a sufficiently small open neighborhood in $U_{3}$ which deformation retracts to $f^{-1}(r)$. Hence, if the stratum $V_{2}'$ is relevant, we arrive at
    	
    	\begin{align*}
    		(Rf \vert_{X \setminus X_{n_{B}-1}})_{\ast}\underline{\mathbb{Q}}_{X \setminus X_{n_{B}-1}} &\simeq 	 
    		h_{3\ast}(R^{0}f \vert_{U'_{2}})_{\ast} \underline{\mathbb{Q}}_{U'_{2} }[-c_{B_3}]
    		\\
    		&\oplus h_{2\ast}'(\tau_{\leq \overline{p}(c_{R_3}) } Rj'_{3\ast}  (R^{0}f \vert_{B \setminus X_{n_{B}-1}})_{\ast} \underline{\mathbb{Q}}_{B \setminus X_{n_{B}-1} }[-c_{B}])\\& \oplus \tau_{\leq \overline{p}(c_{R_3}) } Rj_{3 \ast} \tau_{\leq \overline{p}(c_{R})} Rj_{2\ast}(\underline{\mathbb{Q}}_{Y \setminus R} \oplus \mathcal{L}),
    	\end{align*}
    	
    		where $c_{R_3}:=\operatorname{codim}_{Y}(V_{2}')$ and we employed the commutativity of the diagram
    		\begin{align*}
    			\begin{tikzcd}[ampersand replacement=\&]
    				R \setminus Y_{n_{R}-1} \arrow[r,hookrightarrow,"h_{2}"] \arrow[d,hookrightarrow,"j_{3}'"]	\& Y \setminus Y_{n_{R}-1} \arrow[d,hookrightarrow,"j_{3}"] \\
    				R \setminus Y_{n_{R}-2} \arrow[r,hookrightarrow,"h_{2}'"]	\& Y \setminus Y_{n_{R}-2} 
    			\end{tikzcd}.
    		\end{align*}

    	Otherwise, if the stratum $V_{2}'$ is not relevant, we obtain the quasi-isomorphism
    	\begin{align*}
    		(Rf \vert_{X \setminus X_{n_{B}-1}})_{\ast}\underline{\mathbb{Q}}_{X \setminus X_{n_{B}-1}} 
    		&\oplus h_{2\ast}'(\tau_{\leq \overline{p}(c_{R_3}) } Rj'_{3\ast}  (R^{0}f \vert_{B \setminus X_{n_{B}-1}})_{\ast} \underline{\mathbb{Q}}_{B \setminus X_{n_{B}-1} }[-c_{B}])\\& \oplus \tau_{\leq \overline{p}(c_{R_3}) } Rj_{3 \ast} \tau_{\leq \overline{p}(c_{R})} Rj_{2\ast}(\underline{\mathbb{Q}}_{Y \setminus R} \oplus \mathcal{L}).
    	\end{align*}

    	We repeat the above procedure for each stratum of the refined stratification of $Y$. Let $S_{k}:=Y_{k} \setminus Y_{k-1}$ and $\iota_{k}: \overline{S_{k}} \hookrightarrow Y$ be the canonical closed inclusion. After applying the shift functor $[n]$ to the obtained quasi-isomorphism, we arrive at
    	\begin{align*}
    		Rf_{\ast}\underline{\mathbb{Q}}_{X}[n]\simeq \bigoplus_{S_{k} \in R_{\operatorname{rel}}}\iota_{k \ast}(IC_{\overline{S_{k}}} (R^{0}f\vert_{f^{-1}(S_{k})}\underline{\mathbb{Q}}_{S_{k}})) \oplus IC_{Y}(\underline{\mathbb{Q}}_{Y \setminus R} \oplus \mathcal{L}).
    	\end{align*}
    \end{proof}

\end{document}